\documentclass[reqno]{amsart}

\usepackage{amsmath}
\usepackage{amsfonts,mathrsfs}
\usepackage{amssymb,enumerate}
\usepackage{amsthm}
\usepackage[all]{xy}
\usepackage{hyperref}

\theoremstyle{plain}
\newtheorem{lem}{Lemma}[section]
\newtheorem{cor}[lem]{Corollary}
\newtheorem{prop}[lem]{Proposition}
\newtheorem{thm}[lem]{Theorem}

\newtheorem{intthm}{Theorem}

\theoremstyle{definition}
\newtheorem{defn}[lem]{Definition}
\newtheorem{ex}[lem]{Example}

\newtheorem{disc}[lem]{Remark}

\newtheorem{fact}[lem]{Fact}
\newtheorem{para}[lem]{}

\newcommand{\cat}[1]{\mathcal{#1}}

\newcommand{\catc}{\cat{C}}

\newcommand{\ann}{\operatorname{Ann}}
\newcommand{\mspec}{\operatorname{m-Spec}}

\newcommand{\len}{\operatorname{len}}

\newcommand{\HH}{\operatorname{H}}
\newcommand{\Hom}{\operatorname{Hom}}	
\newcommand{\coker}{\operatorname{Coker}}
\newcommand{\spec}{\operatorname{Spec}}

\newcommand{\tor}{\operatorname{Tor}}

\newcommand{\Ker}{\operatorname{Ker}}

\newcommand{\ideal}[1]{\mathfrak{#1}}
\newcommand{\m}{\ideal{m}}
\newcommand{\n}{\ideal{n}}
\newcommand{\p}{\ideal{p}}

\newcommand{\fm}{\ideal{m}}

\newcommand{\fa}{\ideal{a}}
\newcommand{\fb}{\ideal{b}}
\newcommand{\fc}{\ideal{c}}

\newcommand{\comp}[1]{\widehat{#1}}
\newcommand{\ol}{\overline}

\newcommand{\ass}{\operatorname{Ass}}
\newcommand{\att}{\operatorname{Att}}
\newcommand{\supp}{\operatorname{Supp}}

\newcommand{\Min}{\operatorname{Min}}

\newcommand{\xra}{\xrightarrow}

\newcommand{\onto}{\twoheadrightarrow}
\newcommand{\into}{\hookrightarrow}

\newcommand{\ssm}{\smallsetminus}

\renewcommand{\geq}{\geqslant}
\renewcommand{\leq}{\leqslant}
\renewcommand{\ker}{\Ker}
\renewcommand{\hom}{\Hom}

\newcommand{\Ext}[4][R]{\operatorname{Ext}_{#1}^{#2}(#3,#4)}

\newcommand{\Otimes}[3][R]{#2\otimes_{#1}#3}
\renewcommand{\Hom}[3][R]{\operatorname{Hom}_{#1}(#2,#3)}	
\newcommand{\Tor}[4][R]{\operatorname{Tor}^{#1}_{#2}(#3,#4)}

\newcommand{\Comp}[2]{\widehat{#1}^{\ideal{#2}}}

\newcommand{\md}[1]{#1^{\vee}}
\newcommand{\mdp}[1]{(#1)^{\vee}}
\newcommand{\mdd}[1]{#1^{\vee \vee}}
\newcommand{\mddp}[1]{(#1)^{\vee \vee}}
\newcommand{\bidual}[1]{\delta_{#1}}

\newcommand{\MD}[2]{#1^{\vee(#2)}}
\newcommand{\mcf}{\mathcal{F}}
\newcommand{\mcg}{\mathcal{G}}
\newcommand{\Ga}[2]{\Gamma_{\ideal{#1}}(#2)}
\newcommand{\E}{\operatorname{E}}

\numberwithin{equation}{lem}

\begin{document}

\bibliographystyle{amsplain}

\author{Bethany Kubik}

\address{Bethany Kubik, Department of Mathematical Sciences,
601 Thayer Road \# 222,
West Point, NY 10996
USA}

\email{bethany.kubik@usma.edu}

\urladdr{http://www.dean.usma.edu/departments/math/people/kubik/}

\author{Micah Leamer}

\address{Micah Leamer}

\email{micahleamer@gmail.com}
\author{Sean Sather-Wagstaff}

\address{Sean Sather-Wagstaff, Department of Mathematics,
NDSU Dept \#2750,
PO Box 6050,
Fargo, ND 58108-6050
USA}

\email{Sean.Sather-Wagstaff@ndsu.edu}

\urladdr{http://www.ndsu.edu/pubweb/\~{}ssatherw/}

\thanks{This material is based on work supported by North Dakota EPSCoR and 
National Science Foundation Grant EPS-0814442.
Micah Leamer was supported by a GAANN grant from the Department of Education.
Sean Sather-Wagstaff was supported  by a grant from the NSA}

\title{Homology of artinian and mini-max modules, II}
\date{\today}



\begin{abstract}
Let $R$ be a commutative  ring, and let $L$ and $L'$ be  $R$-modules.
We investigate finiteness conditions (e.g., noetherian, artinian, mini-max, Matlis reflexive)
of the modules $\Ext{i}{L}{L'}$ and $\Tor{i}{L}{L'}$ when $L$ and $L'$ satisfy
combinations of these finiteness conditions. For instance, if $R$ is noetherian,
then given  $R$-modules $M$ and $M'$ such that $M$ is Matlis reflexive and $M'$ is mini-max
(e.g., noetherian or artinian), we prove that
$\Ext{i}{M}{M'}$, $\Ext{i}{M'}{M}$, and $\Tor{i}{M}{M'}$ are Matlis reflexive over $R$ for all $i\geq 0$
and that
$\md{\Ext{i}{M}{M'}}\cong\Tor{i}{M}{M'^\vee}$
and
$\md{\Ext{i}{M'}{M}}\cong\Tor{i}{M'}{M^\vee}$.
\end{abstract}

\maketitle

\section*{Introduction} 
\label{xsec0}

Throughout this paper $R$  denotes a commutative  ring.

It is well-known that, given noetherian $R$-modules $N$ and $N'$, if $R$ is noetherian, 
then $\Ext iN{N'}$ and $\Tor iN{N'}$ are noetherian for all $i$. For other finiteness conditions
(e.g., artinian, mini-max, Matlis reflexive\footnote{See Section~\ref{xsec1} for definitions and 
background material. In particular, Fact~\ref{xpara1} explains the connection between mini-max
and Matlis reflexive modules that is important for our work.}) similar results are not so clear. For instance, 
given artinian $R$-modules $A$ and $A'$, what can one say about 
$\Ext iA{A'}$ and $\Tor iA{A'}$? For Matlis reflexive $R$-modules
$M$ and $M'$, the local case of
the analogous question is treated by Belshoff~\cite{belshoff:scrtmrm}:
if $R$ is local and noetherian, then $\Ext iM{M'}$ and $\Tor iM{M'}$ are Matlis reflexive.

In~\cite{kubik:hamm1} we establish much more general results, still working over a
local noetherian ring. The current paper treats the non-local case, and in some instances
extends results to the non-noetherian setting. 
For instance, the following result is proved in~\ref{para120807a},
\ref{para120807b}, \ref{para120807c}, \ref{para120807d}, and~\ref{para120807e}.

\begin{intthm}\label{intthm120807a}
Assume that $R$ is noetherian. 
Let $A$, $M$, and $M'$ be $R$-modules such that $A$ is artinian, $M$ is mini-max,  $M'$ is Matlis reflexive.
\begin{enumerate}[\rm(a)]
\item \label{intthm120807a1}
Let $\mathcal F$ be a finite subset of $\mspec(R)$ containing $\supp_R(A)\cap\supp_R(M)$,
and set $\fa=\cap_{\m\in \mathcal F}\m$. Then
$\Ext{i}{A}{M}$ is a noetherian $\comp R^{\fa}$-module for all $i\geq 0$. 
\item \label{intthm120807a2}
Let $\fb\subseteq\cap_{\m\in \supp_R(A)\cap\supp_R(L)}\m$. Then for all $i$, the module
$\Tor{i}{A}{M}$ is artinian over $R$ and $\fb$-torsion, hence it is an artinian $\Comp Rb$-module.
\item \label{intthm120807a3}
The $R$-modules $\Ext{i}{M}{M'}$, $\Ext{i}{M'}{M}$, and $\Tor{i}{M}{M'}$ are Matlis reflexive over $R$ for all $i\geq 0$.
\item \label{intthm120807a4}
There are natural $R$-module isomorphisms $\md{\Ext{i}{M}{M'}}\cong\Tor{i}{M}{M'^\vee}$
and
$\md{\Ext{i}{M'}{M}}\cong\Tor{i}{M'}{M^\vee}$
for all $i\geq 0$.
\end{enumerate}
\end{intthm}

One may be dismayed by the technical nature of parts~\eqref{intthm120807a1} and~\eqref{intthm120807a2}
of this result, especially the need to consider a non-canonical completion of $R$. 
However, straightforward examples show that $\Ext{i}{A}{M}$ is not usually noetherian over $R$
or over the completion of $R$ with respect to its Jacobson radical, so this technicality is unavoidable.

It should also be noted that, given the pathological localization behavior of $\Ext iAM$, one cannot simply
localize $\Ext iAM$ and apply the local results of~\cite{kubik:hamm1}.
One needs to apply a more subtle decomposition technique based on a result of Sharp~\cite{sharp:msamaab};
see Fact~\ref{fact120720a}. This result implies that an artinian $R$-module decomposes as a finite direct sum
of $\m$-torsion submodules where $\m$ ranges through the finite set $\supp_R(A)$. 
Such decompositions hold for any $\fb$-torsion module (even over a non-noetherian ring)
when $\fb$ is an intersection of finitely many maximal ideals of $\m$.
Thus, the following result (which is proved in~\ref{para120807f})
applies when $T$ is artinian; it is our substitute for localization that allows us to 
reduce the proof of Theorem~\ref{intthm120807a}\eqref{intthm120807a1} to the local case.

\begin{intthm} \label{xlem100312a}
Assume that $R$ is noetherian,
and let $\fc$ be a finite intersection of maximal ideals of $R$.
Let $T$ and $L$ be $R$-modules such that $T$ is $\fc$-torsion, 
and set $\mcf=\supp_R(T)\cap\supp_R(L)$. 
Let $\mcg$ be a finite set of maximal ideals of $R$ containing $\mcf$. 
Then for all $i\geq 0$ there are $R$-module isomorphisms
\begin{align*}
\textstyle\Ext{i}{T}{L}
& \textstyle
\cong\coprod_{\mathfrak{m}\in \mcg}\Ext[\comp R^{\m}]{i}{\Gamma_\m(T)}{\comp R^{\m}\otimes_{R}L}
\cong\coprod_{\mathfrak{m}\in \mcg}\Ext[R_\m]{i}{T_{\mathfrak{m}}}{L_{\mathfrak{m}}}.
\end{align*}
The second isomorphism is $\Comp Ra$-linear for each ideal 
$\fa\subseteq\cap_{\m\in\mcg}\m$.
Hence, $\Ext iTL$ has an $\Comp Ra$-module structure
that is compatible with its $R$-module structure.
\end{intthm}

Since the decomposition result for artinian modules holds over noetherian and non-noetherian rings
alike, it is reasonable to ask what can be said about these Ext and Tor-modules when $R$ is not noetherian.
The proofs of Theorems~\ref{intthm120807a} and~\ref{xlem100312a} use some techniques that
are inherently noetherian in nature. However, in the case $i=0$ we have the following non-noetherian result,
which we prove in~\ref{para120807h}.
It compliments   
a result of Faith and Herbera~\cite[Proposition 6.1]{faith:ertplcm}  stating that the tensor product of two artinian modules
has finite length.
See also Corollary~\ref{para120807g}.

\begin{intthm} \label{intthm120807b}
Let $A$ and $N$ be  $R$-modules such that $A$ is artinian and $N$ is noetherian. Set $\mcg=\supp_R(A)\cap\ass_R(N)$.
For each $\m\in\mcg$, there is an integer $\alpha_{\m}\geq 0$ such that $\m^{\alpha_{\m}}A=\m^{\alpha_{\m}+1}A$
or $\m^{\alpha_{\m}}\Ga mN=0$; and there is an  isomorphism
$$\textstyle\Hom[R]{A}{N}\cong\coprod_{\m\in\mathcal{G}}\Hom[R]{A/\m^{\alpha_{\m}}A}{(0:_N\m^{\alpha_{\m}})}.$$
In particular, $\Hom AN$ is annihilated by $\cap_{\m\in\mcg}\m^{\alpha_{\m}}$ and has finite length.
\end{intthm}

We summarize the sections of the paper.
Section~\ref{xsec1} contains definitions and background material.
Section~\ref{sec120713a}  consists of foundational material about torsion modules,
and Section~\ref{sec120713b}
does the same for artinian and mini-max modules.
Section~\ref{sec120807a} is devoted to the proof of Theorem~\ref{xlem100312a}
and other similar isomorphism results.
Sections~\ref{xsec2}--\ref{sec110519a} contain the proof of Theorem~\ref{intthm120807a}.
We conclude with Section~\ref{sec120807b} which includes the proof of Theorem~\ref{intthm120807b}
as well as vanishing results for Ext and Tor, including a description of the associated primes of certain Hom-modules.

To conclude this introduction, we mention our mnemonic
naming protocol for modules. It follows in the great tradition of using $I$ for injective modules,
$P$ for projective modules, and $F$ for free or flat modules. Artinian modules are usually named
$A$ or $A'$. Modules with finiteness assumptions on their Bass numbers or Betti numbers are denoted $B$ and $B'$.
We use $M$ and $M'$ for
mini-max (e.g., Matlis
reflexive) modules.
The symbols $N$ and $N'$ are reserved for noetherian modules.
Torsion modules are usually $T$ or $T'$. Modules with no specific properties are 
mostly denoted $L$ and $L'$.

\section{Foundations}\label{xsec1}

This section contains notations, definitions, and other  background material for use throughout the paper.

\begin{defn}
For each ideal $\fa\subseteq R$, let $\Comp Ra$ 
denote the $\fa$-adic completion of $R$,
and set $V(\fa)=\{\p\in\spec(R)\mid\fa\subseteq\p\}$.
Let $\mspec (R)$ denote the set of maximal ideals of $R$.
Given an $R$-module $L$,  let $\E_R(L)$ denote the injective hull of $L$. 
\end{defn}

\begin{fact}\label{fact120719a}
Assume that $R$ is noetherian, and let $\fa$ be an ideal of $R$.
Recall that $\fa\Comp Ra$ is contained in the Jacobson radical of $\Comp Ra$,
and that $\Comp Ra/\fa\Comp Ra\cong R/\fa$; see~\cite[Theorems~8.11 and~8.14]{matsumura:crt}.
From this, it is straightforward to show that there are inverse bijections
$$\xymatrix@R=1mm{
\mspec(R)\cap V(\fa)\ar@{<->}[r] &\mspec(\Comp Ra) \\
\m\ar@{|->}[r]&\m\Comp Ra \\
\n\cap R&\n\ar@{|->}[l]}$$
where $\n\cap R$ denotes the contraction of $\n$ along the natural map $R\to\Comp Ra$.
\end{fact}

\begin{defn} \label{xnotn100204a}
Set
$E_R=\textstyle\coprod_{\m\in \mspec{(R)}}\E_R(R/\m)$.
Let $(-)^{\vee(R)}=\Hom[R]{-}{E_R}$ be the Matlis duality functor.  
We  set $E=E_R$ and $(-)^{\vee}=(-)^{\vee(R)}$
when the ring $R$ is understood.
Set $\mddp{-}=\mdp{\mdp{-}}$ and similarly for $(-)^{\vee(R)\vee(R)}$.
Given  an $R$-module $L$,
the \emph{natural biduality map} for $L$ is the map
$\bidual{L}\colon L\to\mdd{L}$
given by
$\bidual{L}(l)(\psi)=\psi(l)$,
and $L$ is said to be \emph{Matlis reflexive}
if $\bidual{L}$ is an isomorphism.
\end{defn}

\begin{fact}\label{xfact100909a}
Assume that $R$ is noetherian.
Then $E$ is a minimal injective cogenerator for $R$,
that is, for each $R$-module $L$, the natural biduality map
$\bidual{L}\colon L\to\mdd{L}$ is a monomorphism; see~\cite[Exercise 3.3.4]{enochs:rha}.
From this, we have $\ann_R(L)=\ann_R(\md L)$.  Indeed,
the biduality map explains the third containment in the next display; the remaining containments are standard
since $\md{(-)}=\Hom{-}{E}$:
\begin{align*}
\ann_R(L)
&\subseteq\ann_R(\md L)
\subseteq\ann_R(\mdd L)
\subseteq\ann_R(L).
\end{align*}
\end{fact}

\begin{defn}\label{xdefn100317a}
An $R$-module $L$ is said to be \emph{mini-max} if there is a noetherian submodule
$N\subseteq L$ such that the quotient $L/N$ is artinian.
\end{defn}

\begin{fact}[\protect{\cite[Theorem~12]{belshoff:gmd}}]\label{xpara1}  
Assume that $R$ is noetherian.
An $R$-module $L$ is  Matlis reflexive
if and only if $L$ is mini-max and $R/\ann_R(L)$ is semi-local and complete, that is, complete with respect to its Jacobson radical. 
\end{fact}

The proofs of the next three facts given in \cite{kubik:hamm1} also work over non-local rings.

\begin{fact}[\protect{\cite[Lemma 1.23]{kubik:hamm1}}] \label{xlem:extensions}
The class of noetherian (respectively,  artinian or finite length) $R$-modules
is closed under submodules, quotients, and extensions.

Assume that $R$ is noetherian.
The class of mini-max (respectively, Matlis reflexive) $R$-modules
is closed under submodules, quotients, and extensions.
It follows that
the class of mini-max $R$-modules is the smallest class of $R$ modules
containing the artinian and noetherian $R$-modules that is closed under extensions.
See, e.g., the proof of~\cite[Lemma 1.23]{kubik:hamm1}.
\end{fact}

\begin{fact}[\protect{\cite[Lemma 1.24]{kubik:hamm1}}]\label{xlem2of3}
Let $\catc$ be a class  $R$-modules
that is closed under submodules, quotients, and extensions.
\begin{enumerate}[\rm(a)]
\item  \label{xlem100430a2}
Given an exact sequence $L'\xra{f} L\xra{g} L''$,
if $L',L''\in\catc$, then $L\in\catc$.
\item  \label{xlem100430a3}
Given an $R$-complex $X$ 
and  an integer $i$, if $X_i\in\catc$,
then $\HH_i(X)\in\catc$.
\item  \label{xlem100430a4}
Assume that $R$ is noetherian.
Given a noetherian $R$-module $N$, if $L\in\catc$, then 
$\Ext{i}{N}{L},\Tor iNL\in\catc$.
\end{enumerate}
\end{fact}

\begin{fact}[\protect{\cite[Lemma 1.25]{kubik:hamm1}}] \label{xlem100614a}
Let $R\to S$ be a  ring homomorphism, and 
let $\catc$ be a class of $S$-modules
that is closed under submodules, quotients, and extensions.
Fix an $S$-module $L$, an $R$-module $L'$, an $R$-submodule $L''\subseteq L'$, and an index $i\geq 0$.
\begin{enumerate}[\rm(a)]
\item  \label{xlem100614a1}
If $\Ext{i}{L}{L''},\Ext{i}{L}{L'/L''}\in\catc$, then $\Ext{i}{L}{L'}\in\catc$.
\item  \label{xlem100614a2}
If $\Ext{i}{L''}{L},\Ext{i}{L'/L''}{L}\in\catc$, then $\Ext{i}{L'}{L}\in\catc$.
\item  \label{xlem100614a3}
If $\Tor{i}{L}{L''}, \Tor{i}{L}{L'/L''}\in\catc$, then $\Tor{i}{L}{L'}\in\catc$.
\end{enumerate}
\end{fact}

\begin{defn}
\label{xdefn101025a}
A prime ideal $\p$ of $R$ is \emph{associated} to $L$ if there is an $R$-module monomorphism $R/\p\into L$;
the set of primes associated to $L$ is denoted $\ass_R(L)$.
The \emph{support} of an $R$-module $L$ is
$\supp_R(L)=\{\p\in\spec(R)\mid L_{\p}\neq 0\}$.
The set of minimal elements of $\supp_R(L)$ with respect to inclusion
is denoted $\Min_R(L)$.
\end{defn}

\begin{defn}\label{xdefn100207a}
Let $\fa$ be an ideal of $R$, and let $L$ be an $R$-module.
Set
$$\Gamma_{\fa}(L)=\{x\in L\mid\text{$\fa^nx=0$ for $n\gg 0$}\}.$$
We say that $L$ is \emph{$\fa$-torsion} if $L=\Gamma_{\fa}(L)$.
\end{defn}

Here is something elementary and useful.

\begin{lem}\label{lem120719b}
Let $U\subseteq R$ be multiplicatively closed. For all $U^{-1}R$-modules $M$ and $N$, one has
$\Hom[U^{-1}R]MN=\Hom MN$.
\end{lem}

\begin{proof}
Given the natural  inclusion
$\Hom[U^{-1}R]MN\subseteq\Hom MN$, it suffices to verify that each $f\in\Hom MN$ is 
$U^{-1}R$-linear, which we verify  next.
$$\textstyle
f(\frac rum)
=\frac 1uuf(\frac rum)
=\frac 1uf(u\frac rum)=\frac 1uf(rm)=\frac 1urf(m)=\frac ruf(m)
$$
\end{proof}

\begin{fact}\label{xfact101221a}
Assume that $R$ is noetherian, 
and let $\fb$ be an ideal of $R$. For each $\p\in\spec(R)$,
one has
$$\Gamma_{\fb}(\E_R(R/\p))=
\begin{cases}
\E_R(R/\p)&\text{if $\fb\subseteq\p$} \\
0&\text{if $\fb\not\subseteq\p$.}
\end{cases}$$
The point is that $\E_R(R/\p)$ is $\p$-torsion and multiplication by any 
element of $R\ssm \p$ describes an automorphism of $\E_R(R/\p)$.
\end{fact}

\begin{fact}\label{fact110426a}
Assume that $R$ is noetherian, 
and let $U\subseteq R$ be  multiplicatively closed. For each $\p\in\spec(R)$,
one has
$$U^{-1}\E_R(R/\p)=
\begin{cases}
\E_R(R/\p)\cong \E_{U^{-1}R}(U^{-1}R/\p U^{-1}R)&\text{if $\p\cap U=\emptyset$} \\
0&\text{if $\p\cap U\neq \emptyset$.}
\end{cases}$$
See, e.g., \cite[Theorems~3.3.3 and~3.3.8(6)]{enochs:rha}
or~\cite[Theorem 18.4(vi)]{matsumura:crt} or the proof of Lemma~\ref{lem120719a}\eqref{lem120719a2}.
\end{fact}

\begin{defn}\label{xdefn100414a}
Let  $L$ be an $R$-module, $\mathfrak{p}\in\spec{R}$ and $\kappa(\p):=R_{\p}/\p R_{\p}$.
For each integer $i\geq 0$,
the $i$th \emph{Bass number} of $L$ with respect to $\p$ and the $i$th \emph{Betti number} of $L$ with respect to $\p$ are as follows:
\begin{align*}
\mu^i_R(\p,L)&=\dim_{\kappa(\p)}(\Ext[R_{\p}]{i}{\kappa(\p)}{L_{\p}})
&&&\beta_i^R(\p,L)&=\dim_{\kappa(\p)}(\Tor[R_{\p}]{i}{\kappa(\p)}{L_{\p}}).
\end{align*}
When $R$ is quasi-local with maximal ideal $\m$, we abbreviate
$\mu^i_R(L)=\mu^i_R(\m,L)$ and
$\beta_i^R(L)=\beta_i^R(\m,L)$.
\end{defn}

\begin{disc}\label{xdisc101110a}
Let $L$ be an $R$-module. For each $i$ and each $\p\in\spec(R)$,
we have 
\begin{align*}
\mu^i_R(\p,L)&=\mu^i_{R_{\p}}(L_{\p})
&\beta_i^R(\p,L)&= \beta _i^{R_{\p}}(L_{\p}).
\end{align*}
\end{disc}

\begin{disc}\label{xdisc100414a}
Assume that $R$ is noetherian, and let  $L$ be an $R$-module.
\begin{enumerate}[(a)]
\item\label{xdisc100414a2}
If $I$ is a minimal injective resolution of $L$,
then for each index $i\geq 0$ we have
$$\textstyle
I^{i}\cong \coprod_{\p\in\spec(R)}\E_R(R/\p)^{(\mu^i_R(\p,L))}\cong \coprod_{\p\in\supp_R(L)}\E_R(R/\p)^{(\mu^i_R(\p,L))}.$$
See, e.g., \cite[Theorem 18.7]{matsumura:crt}.
\item\label{xdisc100414a3}
For each $\p\in\spec{R}$, the quantity
$\mu^i_R(\p,L)$ is finite for all $i\geq 0$ if and only if $\beta_i^R(\p,L)$ is finite
for all $i\geq 0$; see~\cite[Proposition 1.1]{lescot:spmi} and the localization equalities in Remark~\ref{xdisc101110a}.
\end{enumerate}
\end{disc}

\section{Torsion Modules}\label{sec120713a}

This section  consists of foundational material about torsion modules.
For the next  fact, the proofs  in \cite{kubik:hamm1}  work over non-local non-noetherian rings.

\begin{fact}[\protect{\cite[1.2--1.4]{kubik:hamm1}}]\label{xpara0'''}
Let $\fa$ be an ideal of $R$, and let $L$,  $T$, and $T'$ be  
$R$-modules such that $T$ and $T'$ are $\fa$-torsion.
\begin{enumerate}[\rm(a)] 
\item \label{xitem101021a}
Then $T$ has an $\Comp Ra$-module structure that is compatible with its $R$-module structure
via the natural map $R\to\Comp Ra$.
\item \label{xpara0e}
The natural map $T\to \Comp Ra\otimes_RT$ is an isomorphism of $\Comp Ra$-modules.
\item \label{xpara0g}
The left and right $\Comp Ra$-module structures on $\Comp Ra\otimes_RT$ 
are the same.
\item \label{xpara0h} A set $Z\subseteq T$ is an $R$-submodule if and only if  it is an $\Comp Ra$-submodule.
\end{enumerate}
\end{fact}

The next result contains a non-local version of~\cite[Lemma 1.5]{kubik:hamm1}.

\begin{lem}\label{disc120720a}
Let $\fa$ be an ideal of $R$, and let   $T$ be an 
$\fa$-torsion $R$-module.
\begin{enumerate}[\rm(a)] 
\item \label{xlem100206c1}
If $L$ is an $\Comp Ra$-module (e.g., if $L$ is an $\fa$-torsion $R$-module), then
$\Hom{T}{L}=\Hom[\Comp{R}a]{T}{L}$.
\item \label{xlem100206c2}
If $L'$ is an $R$-module, then there is an $\Comp Ra$-module isomorphism
$\Hom T{L'}\cong\Hom{T}{\Gamma_{\fa}(L')}=\Hom[\Comp Ra]{T}{\Gamma_{\fa}(L')}$.
\end{enumerate}
\end{lem}

\begin{proof}
\eqref{xlem100206c1}
The first isomorphism in the following sequence is Hom-cancellation.
$$\Hom{T}{L}\cong\Hom{T}{\Hom[\Comp Ra]{\Comp Ra}{L}}\cong\Hom[\Comp{R}a]{\Otimes{\Comp Ra}T}{L}\cong\Hom[\Comp{R}a]{T}{L}$$
The second isomorphism is Hom-tensor adjointness, and the third one is from Fact~\ref{xpara0'''}\eqref{xpara0e}.
One checks that these isomorphisms are compatible with the inclusion
$\Hom{T}{L}\supseteq\Hom[\Comp{R}a]{T}{L}$, so this inclusion is an equality.

\eqref{xlem100206c2}
The desired equality follows from part~\eqref{xlem100206c1}.
For the isomorphism,
consider the map $i_*\colon\Hom T{\Gamma_{\fa}(L')}\to\Hom{T}{L'}$, which is induced by the inclusion
$i\colon\Ga a{L'}\into L'$. Since $T$ is $\fa$-torsion,  
it is an $\Comp Ra$-module by Fact~\ref{xpara0'''}\eqref{xitem101021a}.
Using this, it is straightforward to show that $i_*$ is $\Comp Ra$-linear.
The proof of~\cite[Lemma~1.5]{kubik:hamm1} shows that $i_*$ is bijective.
\end{proof}

\begin{lem}\label{lem120818a}
Let $\m\in\mspec(R)$, and let $T$ be an $\m$-torsion $R$-module. 
For each $u\in R\ssm\m$
multiplication by $u$ describes an automorphism of $T$.
\end{lem}

\begin{proof}
The kernel and cokernel of 
the map $T\xra uT$
are $u$-torsion and $\m$-torsion. Hence, they are torsion with respect to $uR+\m=R$, that is, they
are both 0.
\end{proof}

\begin{lem}\label{lem110517a}
Let $\mcf\subseteq\mspec(R)$.
For each $\m\in\mcf$, let $T(\m)$ be an $\m$-torsion $R$-module, and set $T=\coprod_{\m\in\mcf}T(\m)$.
Then 
we have the following.
\begin{enumerate}[\rm(a)]
\item \label{item120718c}
For each $\n\in\mspec(R)$, the composition of natural maps
$\Gamma_{\n}(T)\to T\to T_{\n}$ is an isomorphism.
If $\n\in\mcf$, then the natural map $T(\n)=\Ga n{T(\n)}\to\Ga nT$ is an isomorphism.
Each map is $R_\n$-linear and $\Comp Ra$-linear for any ideal $\fa\subseteq\n$.
\item\label{item120718d}
One has  
\begin{align*}
\Min_R(T)&=\ass_R(T)=\supp_R(T)=\supp_R(T)\cap\mcf\\
&=\{\m\in\mcf\mid T(\m)\neq 0\}=\{\m\in\mcf\mid T_\m\neq 0\}.
\end{align*}
\item \label{item120718e} 
The module $T$ is $\fa$-torsion  for each ideal $\fa\subseteq\cap_{\m\in\supp_R(T)}\m$, and
$$\textstyle\sum_{\mathfrak{m}\in\supp_R(T)}\Gamma_{\mathfrak{m}}(T)=\sum_{\mathfrak{m}\in\mathcal{F}}\Gamma_{\mathfrak{m}}(T)
=T\cong\coprod_{\m\in\mathcal{F}}T_{\m}\cong\coprod_{\m\in\supp_R(T)}T_{\m}.$$
The sum $\sum_{\mathfrak{m}\in\mathcal{F}}\Gamma_{\mathfrak{m}}(T)=\sum_{\mathfrak{m}\in\supp_R(T)}\Gamma_{\mathfrak{m}}(T)$
is a direct sum,
and the isomorphisms are $\Comp Ra$-linear  for each ideal $\fa\subseteq\cap_{\m\in\supp_R(T)}\m$.
\end{enumerate}
\end{lem}

\begin{proof}
Let $\p\in\spec(R)$ and $\m\in\mcf$ and $\n\in\mspec(R)$.
Because each $T(\m)$ is  $\m$-torsion, if $\p\neq \m$ and $\n\neq \m$, then $T(\m)_{\p}=0=\Ga n{T(\m)}$.
(Lemma~\ref{lem120818a} may be helpful here.)
Also,  the natural maps $\Ga m{T(\m)}\to T(\m)\to T(\m)_{\m}$  are bijective.

\eqref{item120718c}
The bijectivity of the given maps (which are at least $R$-linear) follows readily from the previous paragraph.
Since $T(\n)\cong \Ga nT\cong T_\n$ is $\n$-torsion,
Fact~\ref{xpara0'''}\eqref{xitem101021a} implies that $T(\n)$ is an $\Comp Ra$-module for each ideal $\fa\subseteq\n$,
and Lemma~\ref{disc120720a}\eqref{xlem100206c1} tells us that any $R$-module homomorphism
$\Ga nT\to T_\n$ or $T(\n)\to \Ga nT$ is $\Comp Ra$-linear.
It follows that each such map is $\Comp Rn$-linear, so it is $R_\n$-linear by restriction of scalars
along the natural map $R_\n\to\Comp Rn$.

\eqref{item120718d}
The equality in the next sequence is from the previous two paragraphs:
$$\supp_R(T)=\{\m\in\mcf\mid\Ga mT\neq 0\}\subseteq\mcf\subseteq\mspec(R).$$
The containments are by definition.

From the containment $\supp_R(T)\subseteq\mspec(R)$, 
we conclude that each $\m\in\supp_R(T)$ is both maximal
and minimal in $\supp_R(T)$. This explains the equality $\supp_R(T)=\Min_R(T)$,
and the containment $\ass_R(T)\subseteq\supp_R(T)$ is standard.

It remains to show that $\supp_R(T)\subseteq\ass_R(T)$.
Let $\m\in\supp_R(T)$. Part~\eqref{item120718c} implies that $\Ga mT\cong T_{\m}\neq 0$,
so there is a non-zero element $x\in\Ga mT\subseteq T$. This element is $\m$-torsion, so there is an integer $n\geq 0$
such that $\m^{n+1} x=0\neq \m^{n} x$. Any non-zero element $y\in\m^{n} x$ therefore has $\ann_R(y)=\m$,
so $\m\in\ass_R(T)$, as desired.

\eqref{item120718e}
The containment $\supp_R(T)\subseteq\mcf$
from part~\eqref{item120718d}
implies that
$$\textstyle\sum_{\mathfrak{m}\in\supp_R(T)}\Gamma_{\mathfrak{m}}(T)\subseteq\sum_{\mathfrak{m}\in\mathcal{F}}\Gamma_{\mathfrak{m}}(T).$$
The reverse containment follows from the fact that if $\m\in\mcf\ssm\supp_R(T)$, then $\Ga mT\cong T_{\m}=0$
by part~\eqref{item120718c}.
The sum $\sum_{\mathfrak{m}\in\mathcal{F}}\Gamma_{\mathfrak{m}}(T)$
is a direct sum since distinct ideals in $\mcf$ are comaximal.
Since the natural map $T(\m)\to\Ga mT$ is an isomorphism for each $\m\in\mcf$, the equality
$\sum_{\mathfrak{m}\in\mathcal{F}}\Gamma_{\mathfrak{m}}(T)
=T$ now follows. 
The isomorphisms
$\textstyle\sum_{\mathfrak{m}\in\supp_R(T)}\Gamma_{\mathfrak{m}}(T)\cong\coprod_{\m\in\supp_R(T)}T_{\m}$
and
$\textstyle\sum_{\mathfrak{m}\in\mathcal{F}}\Gamma_{\mathfrak{m}}(T)\cong
\coprod_{\m\in\mathcal{F}}T_{\m}$
follow from the directness of the sums, using part~\eqref{item120718c}.

Fix an ideal $\fa\subseteq\cap_{\m\in\supp_R(T)}\m$.
The fact that $T$ is $\fa$-torsion follows readily from the equality $T=\sum_{\mathfrak{m}\in\supp_R(T)}\Gamma_{\mathfrak{m}}(T)$.
The $\Comp Ra$-linearity of each of the isomorphisms in the statement of the result is a consequence of
Lemma~\ref{disc120720a}\eqref{xlem100206c1}.
\end{proof}

The next result provides the prototypical example of a module $T$ as in the previous result.

\begin{lem}\label{lem120727a}
Let $\mcf$ be a finite subset of $\mspec(R)$, and set $\fb=\cap_{\m\in\mcf}\m$.
If $T$ is a $\fb$-torsion $R$-module, then for each $\m\in\mcf$ there is an $\m$-torsion
$R$-module $T(\m)$ such that there is an $\Comp Rb$-module isomorphism
$T\cong\coprod_{\m\in\mcf}T(\m)$.
\end{lem}

\begin{proof}
Fact~\ref{xpara0'''}\eqref{xitem101021a} implies that $T$ is a module over the
product $\Comp Rb\cong\prod_{\m\in\mcf}\Comp Rm$;
this product decomposition comes from the fact that $\mcf$ is finite. 
Furthermore, $T$ is torsion with respect to the Jacobson radical 
$\fb\Comp Rb\subseteq\Comp Rb$.
Using the natural idempotent elements
of $\Comp Rb$, we know that $T$ decomposes as a coproduct
$T\cong\coprod_{\m\in\mcf}T_{\m\Comp Rb}$.
Since $\mcf$ is finite, we have $\fb\Comp Rb_{\m\Comp Rb}=\m\Comp Rb_{\m\Comp Rb}$
for each $\m\in\mcf$. The fact that $T$ is $\fb$-torsion implies that $T_{\m\Comp Rb}$ is
$\m\Comp Rb_{\m\Comp Rb}$-torsion, hence $\m$-torsion.
Thus, we have the desired decomposition with $T(\m)=T_{\m\Comp Rb}$.
\end{proof}

\begin{lem}\label{lem120719a}
Let $\mcf\subseteq\mspec(R)$.
For each $\m\in\mcf$, let $T(\m)$ be an $\m$-torsion $R$-module, and set $T=\coprod_{\m\in\mcf}T(\m)$.
Fix an ideal  $\fa\subseteq R$ and a multiplicatively closed subset $U\subseteq R$,
and set $\mcf_U=\{\m\in\mcf\mid U\cap\m=\emptyset\}$.
Then 
we have the following:
\begin{enumerate}[\rm(a)]
\item \label{lem120719a1}
One has an isomorphism
$\textstyle\Ga aT\cong\coprod_{\m\in\mcf\cap V(\fa)}T(\m)$,
which is $\Comp Rb$-linear for each ideal $\fb\subseteq \fa$,
and 
$$\supp_R(\Ga aT)=\supp_R(T)\cap V(\fa)=\{\m\in\mcf\cap V(\fa)\mid T(\m)\neq 0\}.$$
\item\label{lem120719a2}
One has an isomorphism
$\textstyle U^{-1}T\cong\coprod_{\m\in\mcf_U}T(\m)$.
This isomorphism is $V^{-1}R$-linear for each multiplicatively closed subset $V\subseteq U$.
Also, one has
$$\supp_R(U^{-1}T)=\supp_R(T)\cap \mcf_U=\{\m\in\mcf\mid \text{$U\cap\m=\emptyset$ and $T(\m)\neq 0$}\}.$$
\item \label{lem120719a1'}
If $R$ is noetherian, then one has an isomorphism
$\textstyle \Otimes{\Comp Ra}T\cong\coprod_{\m\in\mcf\cap V(\fa)}T(\m)$,
which is $\Comp Rb$-linear for each ideal $\fb\subseteq \fa$.
\end{enumerate}
\end{lem}

\begin{proof}
\eqref{lem120719a1}
Since each module $T(\m)$ is $\m$-torsion and $\m$ is maximal, Lemma~\ref{lem120818a} can be used to show that
$$\Ga a{T(\m)}=\begin{cases} T(\m)&\text{if $\fa\subseteq\m$}\\
0&\text{if $\fa\not\subseteq\m$.}\end{cases}$$
This explains the  $R$-module isomorphism $\textstyle\Ga aT\cong\coprod_{\m\in\mcf\cap V(\fa)}T(\m)$; 
Lemma~\ref{disc120720a}\eqref{xlem100206c1}
shows that it is also $\Comp Rb$-linear.
The description of $\supp_R(\Ga aT)$ follows from Lemma~\ref{lem110517a}\eqref{item120718d}, 
with a small amount of work.

\eqref{lem120719a2}
We claim that
\begin{equation}\label{lem120719a3} 
U^{-1}T(\m)\cong\begin{cases} T(\m)&\text{if $U\cap\m=\emptyset$}\\
0&\text{if $U\cap\m\neq\emptyset$.}\end{cases}
\end{equation}
If $U\cap\m\neq\emptyset$, say with $u\in U\cap\m$, then $T(\m)$ is $u$-torsion so $U^{-1}T(\m)=0$.
In the case where $U\cap\m=\emptyset$,  the isomorphism $U^{-1}T(\m)\cong T(\m)$
follows from Lemma~\ref{lem120818a}.

The  isomorphism $\textstyle U^{-1}T\cong\coprod_{\m\in\mcf_U}T(\m)$
follows from~\eqref{lem120719a3} as in part~\eqref{lem120719a1}, using 
Lemma~\ref{lem120719b}.
The description of $\supp_R(U^{-1}T)$ follows from Lemma~\ref{lem110517a}\eqref{item120718d}, 
with a little work.

\eqref{lem120719a1'}
Using Facts~\ref{fact120719a} and~\ref{xpara0'''}\eqref{xpara0e}, one see that
$$\Otimes{\Comp Ra}{T(\m)}\cong\begin{cases} T(\m)&\text{if $\fa\subseteq\m$}\\
0&\text{if $\fa\not\subseteq\m$.}\end{cases}$$
This explains the  $\Comp Rb$-isomorphism
$\textstyle \Otimes{\Comp Ra}T\cong\coprod_{\m\in\mcf\cap V(\fa)}T(\m)$, as in part~\eqref{lem120719a1}.
\end{proof}

\begin{lem}\label{wxlem110113a}
Let $\fc$ be an intersection of finitely many maximal ideals of $R$.
Let $U\subseteq R$ be a multiplicatively closed set, and let $T$ be a $\fc$-torsion 
$R$-module. Let 
$\mathcal{F}=\{\m\in\supp_R(T)\mid \m\cap U=\emptyset\}$, and set $V=R\ssm\cup_{\m\in\mathcal{F}}\m$ 
and $\fb=\cap_{\m\in\mathcal{F}}\m$. 
Then there are $R$-module isomorphisms 
$$\textstyle U^{-1}T\cong V^{-1}T\cong \Gamma_{\fb}(T)\cong\coprod_{\m\in\mathcal{F}}T_\m$$
and  $\supp_R(U^{-1}T)=\mcf$.
\end{lem}

\begin{proof}
Note that we have $U\subseteq V$, so
Lemmas~\ref{lem120727a} and~\ref{lem120719a}\eqref{lem120719a2}
provide the isomorphisms $U^{-1}T\cong\coprod_{\m\in\mcf} T_{\m}\cong V^{-1}T$ 
and the equality $\supp_R(U^{-1}T)=\mcf$,
and we have $\Gamma_{\fb}(T)\cong\coprod_{\m\in\mathcal{F}}T_\m$
by Lemmas~\ref{lem120727a} and~\ref{lem120719a}\eqref{lem120719a1}.
\end{proof}

The next two results are proved like Lemma~\ref{wxlem110113a} and~\cite[Lemma 3.7]{kubik:hamm1}.

\begin{lem}\label{lem110617b}
Let $\fc$ be an intersection of finitely many maximal ideals of $R$.
Let $\fa$ be an ideal of $R$, and let $T$ be a $\fc$-torsion $R$-module.
Set $\mcf=V(\fa)\cap\supp_R(T)$, $\fb=\cap_{\m\in\mcf}\m$, and $U=R\ssm\cup_{\m\in\mcf}\m$.
Then we have
$$\textstyle\sum_{\m\in\mcf}\Ga mT=\Ga aT=\Ga bT\cong
U^{-1}T.$$
The sum is a direct sum,  
and we have $\supp_R(\Ga bT)=\mcf$.
\end{lem}

\begin{lem}\label{xlem:tensor}
Let $\fa$  be an ideal of $R$.  Let $L$ and $T$ be $R$-modules such that $T$ is $\fa$-torsion and $\fa^nL=\fa^{n+1}L$ for some $n\geq 0$. 
Then $T\otimes_R(\fa^nL)=0$ and 
$$T\otimes_R L\cong T\otimes_R (L/\fa^nL)\cong (T/\fa^nT)\otimes_R (L/\fa^nL).$$
\end{lem}

\begin{lem}\label{lem110413a} 
Let $\fa$ be an ideal of $R$,
and let $T$ be an $\fa$-torsion $R$-module.
Then $T$ is a noetherian  (respectively, artinian or mini-max) as an $R$-module if and only if it is noetherian 
(respectively, artinian or mini-max) as an $\Comp Ra$-module.
\end{lem}

\begin{proof}
Fact~\ref{xpara0'''}\eqref{xitem101021a} implies that $T$ is an $\Comp Ra$-module.

Since the $R$-submodules of $T$ and the $\Comp Ra$-submodules of $T$ are the same 
by Fact~\ref{xpara0'''}\eqref{xpara0h},
they satisfy the descending chain condition simultaneously, and the artinian case follows.  
Similarly for the noetherian case.

For the mini-max case, suppose that there is an exact sequence of $R$-module homomorphisms $0\to A\to T\to N\to 0$.
Since $T$ is $\fa$-torsion, so are $A$ and $N$. Lemma~\ref{disc120720a}\eqref{xlem100206c1} implies that the given
sequence consists of $\Comp Ra$-module homomorphisms. Since $A$ is artinian over $R$ if and only if it is artinian over $\Comp Ra$,
and $N$ is noetherian over $R$ if and only if it is noetherian over $\Comp Ra$,
it follows that $T$ is mini-max over $R$ if and only if it is mini-max over $\Comp Ra$.
\end{proof}

\begin{lem}\label{xlem101021a}
Assume that $R$ is noetherian, 
and fix an ideal $\fa\subseteq R$.
For each $\p\in V(\fa)$ we have
\begin{equation}\label{xlem101021a3}
\E_{\Comp Ra}(\Comp Ra/\p\Comp Ra)
\cong
\E_R(R/\p)
\cong
\E_{\Comp Rp}(\Comp Rp/\p\Comp Rp).
\end{equation}
The first isomorphism is $\Comp Ra$-linear, and the second one is $\Comp Rp$-linear.
Also there are $\Comp Ra$-module isomorphisms
\begin{equation}\label{xlem101021a1}
\textstyle E_{\Comp Ra}\cong\coprod_{\m\in\mspec(R)\cap V(\fa)}\E_R(R/\m)\cong\Ga a{E_R}.
\end{equation}
In particular, the module $E_{\Comp Ra}$ is $\fa$-torsion.
\end{lem}

\begin{proof}
Let $\p\in V(\fa)$.
Since
$R/\p$ and $\E_R(R/\p)$ are $\p$-torsion, they are $\fa$-torsion,
so they have natural $\Comp Ra$-module structures.
Moreover, $R/\p\subseteq \E_R(R/\p)$ is an $\Comp Ra$-submodule by Fact~\ref{xpara0'''}\eqref{xpara0h}, and
Fact~\ref{fact120719a} shows that $\Comp Ra/\p\Comp Ra\cong R/\p$.
Note that this isomorphism is $\Comp Ra$-linear by Lemma~\ref{disc120720a}\eqref{xlem100206c1}
since the modules in question are $\fa$-torsion.

Claim: The essential extensions of $R/\p$ as an $R$-module are exactly the essential extensions of $R/\p$ as an $\Comp Ra$-module.
First, let $L$ be an essential extension of $R/\p$ as an $R$-module. 
Since $\E_R(R/\p)$ is $\fa$-torsion and is a maximal essential extension of $R/\p$ it follows that $L$ is $\fa$-torsion.  
By Fact~\ref{xpara0'''}\eqref{xitem101021a}, $L$ is an $\Comp{R}{\fa}$-module. 
Let $L'\subseteq L$ be a non-zero $\Comp{R}{\fa}$-submodule.
By restriction of scalars, $L'$ is an $R$-module. Since $L$ is essential as an $R$-module, we have $L'\cap R/\p\neq 0$. 
Thus $L$ is an essential extension of $R/\p$ as an $\Comp{R}{\fa}$-module.  
A similar argument shows that any essential extension of $R/\p\cong \Comp{R}{\fa}/\p\Comp{R}{\fa}$ as an 
$\Comp{R}{\fa}$-module is also an essential extension as an $R$-module.  

From the claim, it follows that the maximal 
essential extensions of $R/\p$ as an $R$-module are exactly the maximal essential extensions of $R/\p$ as an $\Comp Ra$-module,
so $\E_{\Comp{R}{\fa}}(R/\p)\cong \E_R(R/\p)$. 
This isomorphism is $\Comp{R}{\fa}$-linear by Lemma~\ref{disc120720a}\eqref{xlem100206c1}.

Since $\p$ is an arbitrary element of $V(\fa)$, the special case $\fa=\p$ shows that
$\E_R(R/\p)
\cong
\E_{\Comp Rp}(\Comp Rp/\p\Comp Rp)$
so we have the second isomorphism in~\eqref{xlem101021a3}.
The first isomorphism in~\eqref{xlem101021a1} now follows from Fact~\ref{fact120719a} and~\eqref{xlem101021a3}.
Lemma~\ref{lem120719a}\eqref{lem120719a1} explains the second isomorphism in~\eqref{xlem101021a1}.
\end{proof}

The final result of this section compares to part of~\cite[Lemma 1.5(a)]{kubik:hamm1}.

\begin{lem} \label{lem110413b} 
Assume that $R$ is noetherian.
Let $\fa$ be an ideal of $R$, and let $T$  be   an $\fa$-torsion
$R$-module.
Then there is an  isomorphism
$\MD TR\cong\MD{T}{\Comp Ra}$ that is $\Comp Rb$-linear for all ideals $\fb\subseteq\fa$.
\end{lem}

\begin{proof}
This is a consequence of the next display
\begin{align*}
\Hom{T}{E}
\cong\Hom[\Comp Ra]{T}{\Gamma_{\fa}(E)}
\cong\Hom[\Comp Ra]{T}{E_{\Comp Ra}}
\end{align*}
which follows from Lemma~\ref{disc120720a}\eqref{xlem100206c2} with Lemma~\ref{xlem101021a}.
\end{proof}

\section{Artinian and Mini-Max Modules}\label{sec120713b}

We begin this section with an important observation of Sharp~\cite{sharp:msamaab}.

\begin{fact}\label{fact120720a}
Let $A$ be an artinian $R$-module.
By~\cite[Proposition 1.4]{sharp:msamaab}, there is a finite set $\mcf$ of maximal ideals of $R$ such that $A$ is the internal direct sum
$A=\sum_{\m\in\mcf}\Ga{m}{A}$. 
Consequently, Lemmas~\ref{lem110517a} and~\ref{lem120719a} apply to the module $T=A$.
In particular, any localization $U^{-1}A$ is naturally a submodule of $A$
by Lemma~\ref{lem120719a}\eqref{lem120719a2}
so it is artinian over $R$ and hence over $V^{-1}R$ for each multiplicatively closed subset $V\subseteq U$.
Furthermore, any torsion submodule $\Ga bA$ is naturally a submodule of $A$
by Lemma~\ref{lem110517a}\eqref{item120718e} and~\ref{lem120719a}\eqref{lem120719a1}
so it is artinian over $R$ and hence over $\Comp Ra$ for each ideal $\fa\subseteq\fb$.
If $R$ is noetherian, then any torsion submodule $\Ga bA\cong\Otimes{\Comp Rb}{A}$ is naturally a submodule of $A$
by parts~\eqref{lem120719a1} and~\eqref{lem120719a1'} of Lemma~\ref{lem120719a}
so it is artinian over $R$ and hence over $\Comp Ra$ for each ideal $\fa\subseteq\fb$.
\end{fact}

\begin{lem} \label{zlem100207a}
Let $L$ be an $R$-module.
Then $L$ is artinian  if and only if 
$\supp_R(L)$ is a finite set 
and $L_{\p}$ is artinian over $R_{\p}$
for each $\p\in\supp_R(L)$.
\end{lem}

\begin{proof}
The forward implication follows from  
Lemma~\ref{lem110517a}\eqref{item120718d} and Fact~\ref{fact120720a}.

For the reverse implication, assume that   $\supp_R(L)=\{\p_1,\ldots,\p_h\}$,
and that $L_{\p_i}$ is artinian over $R_{\p_i}$
for $i=1,\ldots,h$.
Let $L=L_0\supseteq L_1\supseteq L_2\supseteq \cdots$ be a descending chain of $R$-modules. 
Since
$L_{\p_i}=(L_0)_{\p_i}\supseteq (L_1)_{\p_i}\supseteq (L_2)_{\p_i}\supseteq \cdots$
stabilizes for $i=1,\hdots, h$, we may choose $j\in\mathbb{N}$ so that
$(L_{j})_{\p_i}=(L_{j+n})_{\p_i}$ for $i=1,\ldots,h$ and for all $n\in\mathbb{N}$.
For each $\p\in\spec(R)\smallsetminus\supp_R(L)$, we have
$L_{\p}=0$, so $(L_{j})_{\p}=(L_{j+n})_{\p}$ for all $n\in\mathbb{N}$.
Hence, we have $L_{j}=L_{j+n}$ for all $n\in\mathbb{N}$, and $L$ is artinian.
\end{proof}

Now we talk about another class of modules, motivated by Fact~\ref{xpara1}.

\begin{lem} \label{xlem100409a1}
Fix  an $R$-module $L$ such that $R/\ann_R(L)$
is semi-local and complete. 
\begin{enumerate}[\rm(a)]
\item\label{xlem100409a1a}
The set $\mspec(R)\cap\supp_R(L)=\mspec(R)\cap V(\ann_R(L))$ is finite
and naturally in bijection with $\mspec(R/\ann_R(L))$.
\item\label{xlem100409a1b}
If $R$ is noetherian, then 
$\mspec(R)\cap\supp_R(L)=\mspec(R)\cap\supp_R(\md L)$.
\end{enumerate}
\end{lem}

\begin{proof}
\eqref{xlem100409a1a}
Set $\ol R=R/\ann_R(L)$. We  
assume $L\neq 0$. Let $\pi\colon R\to  \ol R$ be the natural surjection and $\pi^{*}:\spec(\ol R)\to\spec(R)$ the induced map given by $\pi^*(\p)=\pi^{-1}(\p)$.
Since $L_\p=0$ for all $\p$ not containing $\ann_R(L)$ we get  $\supp_R(L)=\pi^*(\supp_{\ol R}(L))$. Therefore,  
$\mspec(R)\cap\supp_R(L)={\pi^*(\mspec(\ol R)\cap\supp_{\ol R}(L))}$.

The ring $\ol R\neq 0$ is semi-local and complete, so it is a finite product of non-zero complete local rings,
say $\ol R\cong\prod_{i=1}^nR_i$. Since $L$ is an $\ol R$-module we have $L=\prod_{i=1}^n L_i$  where $L_i$ is an $R_i$-module. By construction $\ann_{\ol R}(L)=0$, so $L_i\neq 0$ for all $i$. Thus $\mspec(\ol R)\subseteq \supp_{\ol R}(L)$.  This explains the second equality in the following display. The last equality is standard.  
\begin{align*}
\mspec(R)\cap\supp_R(L)&=\pi^*(\mspec(\ol R)\cap\supp_{\ol R}(L))\\
&=\pi^*(\mspec(\ol R))\\
&=\mspec(R)\cap V(\ann_R(L))
\end{align*}
As $\ol R$ is semi-local, this set is finite.

\eqref{xlem100409a1b}
Assume that $R$ is noetherian.  Fact~\ref{xfact100909a} implies that the ring
$R/\ann_R(\md L)=R/\ann_R(L)$ is semi-local and complete,
so part~\eqref{xlem100409a1a} implies that
\begin{align*}
\mspec(R)\cap\supp_R(\md L)
&=\mspec(R)\cap V(\ann_R(\md L)) \\
&=\mspec(R)\cap V(\ann_R(L))\\
&=\mspec(R)\cap\supp_R(L)
\end{align*}
completing the proof.
\end{proof}

The next result compares directly with Fact~\ref{xpara0'''} and Lemma~\ref{disc120720a}\eqref{xlem100206c1}. 

\begin{lem} \label{xlem100409a}
Assume that $R$ is noetherian.
Let $L$  be an $R$-module such that $R/\ann_R(L)$
is semi-local and complete. Set $\fb=\cap_{\m\in\mspec(R)\cap\supp_R(L)}\m$,
and let $\fa\subseteq\fb$.
\begin{enumerate}[\rm(a)]
\item \label{xlem100409a2}
$L$ has an $\Comp Ra$-module structure that is compatible with its $R$-module structure
via the natural map $R\to\Comp Ra$. 
\item \label{item120721a}
The natural map $L\to \Comp Ra\otimes_RL$ is an isomorphism of $\Comp Ra$-modules.
\item \label{item120721b}
The left and right $\Comp Ra$-module structures on $\Comp Ra\otimes_RL$ 
are the same.
\item \label{xlem100409a3}
A subset $Z\subseteq L$ is an $R$-submodule if and only if it is an $\Comp Ra$-submodule.
\item \label{item120721c}
If $L'$ is an $\Comp Ra$-module (e.g., $L'$ is an $\fa$-torsion $R$-module), then
$\Hom{L}{L'}=\Hom[\Comp{R}a]{L}{L'}$.
\end{enumerate}
\end{lem}

\begin{proof}
Assume without loss of generality that $L\neq 0$.
The fact that $R$ is noetherian implies that
$\Comp Ra/\ann_R(L)\Comp Ra$ is isomorphic to the $\fa$-adic completion of $R/\ann_R(L)$.

\eqref{xlem100409a2}
There is a commutative diagram of ring homomorphisms
\begin{equation}
\label{eq120721a}
\begin{split}
\xymatrix{
R\ar[r]\ar[d]
&\Comp Ra\ar[d]\\
R/\ann_R(L)\ar[r]^-\cong
&\Comp Ra/\ann_R(L)\Comp Ra.}
\end{split}
\end{equation}
The map in the bottom row is an isomorphism because
$R/\ann_R(L)$ is semi-local and complete with Jacobson radical
$\fb/\ann_R(L)$; this uses Lemma~\ref{xlem100409a1}\eqref{xlem100409a1a}.
Since $L$ has an $R/\ann_R(L)$-module structure that is compatible with its $R$-module structure
via the natural map $R\to R/\ann_R(L)$, the isomorphism in the bottom row shows that
$L$ has a compatible $\Comp Ra/\ann_R(L)\Comp Ra$-module structure.
It follows that $L$ has a compatible $\Comp Ra$-module structure.

\eqref{item120721a}--\eqref{item120721b}
Diagram~\eqref{eq120721a}
shows that 
$$L\cong \Otimes{(R/\ann_R(L))}{L}\cong \Otimes{(\Comp Ra/\ann_R(L)\Comp Ra)}{L}
\cong \Otimes{\Comp Ra}{L}$$
and the desired conclusions follow readily.

\eqref{xlem100409a3}
The subset $Z\subseteq L$ is an $R$-submodule if and only if it is an $R/\ann_R(L)$-submodule.
The isomorphism in  diagram~\eqref{eq120721a}
shows that $Z$ is an $R/\ann_R(L)$-submodule
if and only if it is an $\Comp Ra/\ann_R(L)\Comp Ra$-submodule,
that is, if and only if it is an $\Comp Ra$-submodule.

\eqref{item120721c}
This is proved like Lemma~\ref{disc120720a}\eqref{xlem100206c1} using part~\eqref{item120721a}.
\end{proof}

The next two results compare directly with  Lemma~\ref{lem110413a} and~\cite[Lemma 1.20]{kubik:hamm1}. 

\begin{lem} \label{lem120721a}
Assume that $R$ is noetherian.
Let $L$ be an $R$-module such that $R/\ann_R(L)$
is semi-local and complete. Set $\fb=\cap_{\m\in\mspec(R)\cap\supp_R(L)}\m$,
and let $\fa\subseteq\fb$.
Then $L$ is  noetherian (respectively, artinian) over $R$ if and only if it is noetherian 
(respectively, artinian) over $\Comp Ra$.
\end{lem}

\begin{proof}
From Lemma~\ref{xlem100409a}\eqref{xlem100409a3} we have $\{R$-submodules of $L\}=\{ \Comp Ra$-submodules of $L\}$.
Thus, the first set satisfies the ascending (respectively, descending) chain condition  if and only if the second one does. 
\end{proof}

\begin{lem} \label{xlem100409a'}
Assume that $R$ is noetherian, and let $L$ be an $R$-module such that $R/\ann_R(L)$
is semi-local and complete. Set $\fb=\cap_{\m\in\mspec(R)\cap\supp_R(L)}\m$,
and let $\fa\subseteq\fb$.
Then the following conditions are equivalent:
\begin{enumerate}[\rm(i)]
\item \label{xlem100409a'6}
$L$ is mini-max as an $R$-module;
\item \label{xlem100409a'7}
$L$ is mini-max as an $\Comp Ra$-module;
\item \label{xlem100409a'8}
$L$ is Matlis reflexive as an $R$-module; and
\item \label{xlem100409a'9}
$L$ is Matlis reflexive as an $\Comp Ra$-module.
\end{enumerate}
\end{lem}

\begin{proof}
Assume without loss of generality that $L\neq 0$.

\eqref{xlem100409a'6}$\iff$\eqref{xlem100409a'7}
Let $Z\subset L$ be a subset. Lemma~\ref{xlem100409a}\eqref{xlem100409a3} says that 
$Z$ is an $R$-submodule
if and only if it is an $\Comp Ra$-submodule.
Assume that $Z$ is an $R$-submodule of $L$.
Lemma~\ref{lem120721a} shows that
$Z$ is noetherian as an $R$-module if and only if it is noetherian as an $\Comp Ra$-module,
and  the quotient $L/Z$ is artinian over $R$ if and only
if it is artinian over $\Comp Ra$. 

\eqref{xlem100409a'6}$\iff$\eqref{xlem100409a'8}
This is an immediate consequence of Fact~\ref{xpara1}.

\eqref{xlem100409a'7}$\iff$\eqref{xlem100409a'9}
The fact that $R$ is noetherian and $R/\ann_R(L)$ is complete explains the isomorphism in the next display
$$R/\ann_R(L)\cong \comp R^{\fa}/\ann_R(L)\comp R^{\fa}\onto \comp R^{\fa}/\ann_{\comp R^{\fa}}(L).$$
The epimorphism comes from the containment $\ann_R(L)\comp R^{\fa}\subseteq\ann_{\comp R^{\fa}}(L)$. 
Thus, the fact that $R/\ann_R(L)$ is semi-local and complete
implies that  
$\comp R^{\fa}/\ann_{\comp R^{\fa}}(L)$  is semi-local and complete. 
Hence, the equivalence \eqref{xlem100409a'7}$\iff$\eqref{xlem100409a'9} is a consequence of Fact~\ref{xpara1}.
\end{proof}

\begin{lem} \label{xfact100317b}
Assume that $R$ is noetherian.
Let $M$ be a mini-max $R$-module and let $U\subseteq R$ be multiplicatively closed.
Then $U^{-1}M$ is a mini-max $U^{-1}R$-module and 
the quantities $\mu^i_R(\p,M),\beta_i^R(\p,M)$ are finite for all $i\geq 0$ and all $\p\in\spec(R)$.  
\end{lem}

\begin{proof}
The claim that $U^{-1}M$ is a mini-max $U^{-1}R$-module follows from the fact that localization is exact and localizing a noetherian (artinian) $R$-module at $U$ yields a noetherian (artinian) 
$U^{-1}R$-module; see Fact~\ref{fact120720a}. Therefore, the remaining conclusions follow from the local case,
using the localization behavior of Bass and Betti numbers from Remark~\ref{xdisc101110a}; see \cite[Lemma 1.19]{kubik:hamm1}.  
\end{proof}

Our next result compares to part of~\cite[Lemma 1.21]{kubik:hamm1}.

\begin{lem}\label{xlem100213b}
Assume that $R$ is noetherian.
Let $L$ be an $R$-module such that $R/\ann_R(L)$ is artinian. Then the following conditions are equivalent:
\begin{enumerate}[\rm(i)]
\item\label{xlem100213b1}
$L$ is mini-max over $R$;
\item\label{xlem100213b2}
$L$ is Matlis reflexive over $R$;
\item\label{xlem100213b5}
$L$ has finite length over $R$.
\end{enumerate}
\end{lem}

\begin{proof}
The implication
$\eqref{xlem100213b5}\implies\eqref{xlem100213b1}$
is routine, and the equivalence $\eqref{xlem100213b1}\iff\eqref{xlem100213b2}$ is from Fact~\ref{xpara1}.

$\eqref{xlem100213b1}\implies\eqref{xlem100213b5}$
Assume that $L$ is mini-max. Then 
$L$ is mini-max as an $R/\ann_R(L)$-module. Over an artinian ring every indecomposable injective module has finite 
length and the prime spectrum is a finite set. By Remark~\ref{xdisc100414a}\eqref{xdisc100414a2}
and Lemma~\ref{xfact100317b} the injective hull of $L$ as an $R/\ann_R(L)$-module is a finite direct sum 
of indecomposable injective $R/\ann_R(L)$-modules. Thus, $L$ injects into a finite length module. Hence $L$ has finite length.
\end{proof}

\begin{lem}\label{xlem101202a}
Assume that $R$ is noetherian, and let $\fa$ be an ideal of $R$.
If $M$ is a mini-max $R$-module, then $\Otimes{\Comp Ra}{M}$ is a mini-max $\Comp Ra$-module.
\end{lem}

\begin{proof}
Let $M$ be mini-max over $R$, and fix an exact sequence of $R$-module homomorphisms
$0\to N\to M\to A\to 0$
where $N$ is noetherian over $R$ and $A$ is artinian over $R$.
The ring $\Comp Ra$ is flat over $R$ since $R$ is noetherian, 
so the base-changed sequence
$$0\to \Otimes{\Comp Ra}{N}\to \Otimes{\Comp Ra}{M}\to \Otimes{\Comp Ra}{A}\to 0$$
is an exact sequence of $\Comp Ra$-module homomorphisms.
The $\Comp Ra$-module $\Otimes{\Comp Ra}{N}$ is noetherian.
Fact~\ref{fact120720a} implies that the $\Comp Ra$-module $\Otimes{\Comp Ra}{A}$ is artinian,
so $\Otimes{\Comp Ra}{M}$ is  mini-max over $\Comp Ra$.
\end{proof}

\section{Isomorphisms for $\Ext{i}{T}{L}$}
\label{sec120807a}

This section contains the proof of Theorem~\ref{xlem100312a} (in~\ref{para120807f})
and other  isomorphism results that are used in later sections.

\begin{lem} \label{xlem101120a}
Assume that $R$ is noetherian.
Let $I$ be an injective $R$-module, and $\mcg$  a finite subset of $\mspec(R)$.
Set $\fb=\cap_{\m\in \mathcal{G}}\m$ and
$V=R\ssm\cup_{\m\in \mathcal{G}}\m$, and let $U\subseteq V$ be  multiplicatively closed.
Then the natural map $\Ga bI\to\Ga{b}{U^{-1}I}$ is bijective.
\end{lem}

\begin{proof}
Write $I=\coprod_{\p\in\spec(R)}\E_R(R/\p)^{(\mu_p)}$.
By Fact~\ref{fact110426a} and Remark~\ref{xdisc100414a}\eqref{xdisc100414a2}, 
the natural map $\rho\colon I\to U^{-1}I$ is a split surjection with
$\ker(\rho)=\coprod_{\p\cap U\neq\emptyset}\E_R(R/\p)^{(\mu_p)}$.
Since $\rho$ is a split surjection, it follows that
$\Ga{b}{\rho}\colon \Ga{b}{I}\to \Ga{b}{U^{-1}I}$ is a split surjection with
$\ker(\Ga{b}{\rho})=\coprod_{\p\cap U\neq\emptyset}\Ga{b}{\E_R(R/\p)}^{(\mu_p)}$.
Thus, it remains to show that $\Ga{b}{\E_R(R/\p)}=0$ when $\p\cap U\neq\emptyset$.

Assume that $\p\cap U\neq\emptyset$. Then $\p\cap V\neq\emptyset$, so $\p\nsubseteq\m$ for all $\m\in\mcg$. 
Since $\mcg$ is a set of maximal ideals it follows that $\m\nsubseteq\p$ for all $\m\in\mcg$.
Hence, we have $\fb=\cap_{\m\in \mathcal{G}}\m\nsubseteq\p$
since $\mcg$ is finite. 
Fact~\ref{xfact101221a} implies that
$\Ga{b}{\E_R(R/\p)}=0$ and the result follows.
\end{proof}

\begin{para}[Proof of Theorem~\ref{xlem100312a}]\label{para120807f}
Let $I$ be a minimal injective resolution of $L$. If $\p\in\spec(R)\ssm\supp_R(L)$, then
the condition $L_{\p}=0$ implies that $\E_R(R/\p)$ does not occur as a summand of any $I^j$; see 
Remarks~\ref{xdisc101110a} and~\ref{xdisc100414a}\eqref{xdisc100414a2}.
For all $\m\in\mspec(R)\ssm\mcf$, either 
$\m\notin\supp_R(T)$ or $\m\notin\supp_R(L)$, so either
$T_\m=0$ or $\Gamma_\m(I)=0$ by the above remark.  
Note that $T_{\m}$ is $\m$-torsion since either $T_{\m}=0$ or $\fc R_{\m}=\m R_{\m}$.
Thus, Lemma~\ref{disc120720a}\eqref{xlem100206c2} implies that 
$\Hom{T_\m}{I}\cong\Hom{T_\m}{\Gamma_\m(I)}=0$ for all $\m\notin\mcf$. Since $\supp_R(T)$ and $\mcg$ both contain $\mathcal{F}$, this explains step (3) in the next display
\begin{align*}
\hom{T}{I}
&\textstyle\stackrel{(1)}\cong \hom{\coprod_{\m\in\supp_R(T)}T_\m}{I}\\
&\textstyle\stackrel{(2)}\cong \coprod_{\m\in\supp_R(T)}\hom{T_\m}{I}\\
&\textstyle\stackrel{(3)}\cong \coprod_{\m\in\mcg}\hom{T_\m}{I}\\
&\textstyle\stackrel{(4)}\cong \coprod_{\m\in\mcg}\hom{T_\m}{\Gamma_{\m}(I)}\\
&\textstyle\stackrel{(5)}\cong \coprod_{\m\in\mcg}\hom{T_\m}{\Gamma_{\m}(I_\m)}\\
&\textstyle\stackrel{(6)}= \coprod_{\m\in\mcg}\hom{T_\m}{I_\m} \\
&\textstyle\stackrel{(7)}= \coprod_{\m\in\mcg}\hom[R_\m]{T_\m}{I_\m}.
\end{align*}
Step (1) comes from Lemmas~\ref{lem110517a}\eqref{item120718e} and~\ref{lem120727a},
and (5) is from Lemma~\ref{xlem101120a}.
Step (2) is standard, since $\supp_R(T)$ is finite. 
Lemma~\ref{disc120720a}\eqref{xlem100206c2} and
Facts~\ref{xfact101221a}--\ref{fact110426a} explain steps (4) and (6), respectively, and
step (7) is from Lemma~\ref{lem120719b}.
Since $I_{\m}$ is an $R_{\m}$-injective resolution of $L_{\m}$,
it follows that $\Ext{i}{T}{L}\cong\coprod_{\mathfrak{m}\in \mcg}\Ext[R_\m]{i}{T_\m}{L_\m}$.  

Set $\fb=\cap_{\m\in\mcg}\m$.
For each $\m\in\mspec(R)$, the module $T_\m\cong\Ga mT$ is $\m$-torsion by
Lemmas~\ref{lem110517a}\eqref{item120718c} and~\ref{lem120727a}.
Thus $T_\m$ is an
$\Comp Rm$-module, and so is $\Ext[R_{\m}]{i}{T_{\m}}{L_{\m}}$.
Thus, the coproduct 
$\Ext{i}{T}{L}
\cong\coprod_{\mathfrak{m}\in \mcg}\Ext[R_{\m}]{i}{\Gamma_\m(T)}{L_\m}$
is  a module over the product
$\Comp Rb=\prod_{\m\in\mcg}\Comp Rm$ using componentwise multiplication.
By restriction of scalars, this is also a module over $\Comp Ra$ for each $\fa\subseteq\fb$.

The first 
$\Comp Rm$-module isomorphism in the next display is from~\cite[Lemma 4.2]{kubik:hamm1}
$$\Ext[R_{\m}]{i}{T_{\m}}{L_{\m}}
\cong \Ext[\comp R^{\m}]{i}{T_{\m}}{\comp R^{\m}\otimes_{R_\m}L_{\m}}
\cong\Ext[\comp R^{\m}]{i}{\Gamma_\m(T)}{\comp R^{\m}\otimes_{R}L}.$$ 
The second isomorphism is from Lemmas~\ref{lem110517a}\eqref{item120718c} and~\ref{lem120727a}, and
using the standard isomorphism $\comp R^{\m}\otimes_{R_\m}L_{\m}\cong \comp R^{\m}\otimes_{R}L$. Since these isomorphisms are $\Comp Rm$-linear for each $\m$,
the induced isomorphism on coproducts
$\coprod_{\mathfrak{m}\in \mcg}\Ext[\comp R^{\m}]{i}{\Gamma_\m(T)}{\comp R^{\m}\otimes_{R}L}
\cong\coprod_{\mathfrak{m}\in \mcg}\Ext[R_\m]{i}{T_{\mathfrak{m}}}{L_{\mathfrak{m}}}$
is linear over the product
$\Comp Rb=\prod_{\m\in\mcg}\Comp Rm$ hence over $\Comp Ra$ for each $\fa\subseteq\fb$.
\qed
\end{para}

In the next result, one can take $\fa=\cap_{\m\in\mcg}\m$, for instance.

\begin{thm} \label{xlem100312az}
Assume that $R$ is noetherian, and let $\fc$ be an intersection of finitely many maximal ideals of $R$.
Let $T$ and $L$ be $R$-modules such that $T$ is $\fc$-torsion, and set $\mcf=\supp_R(T)\cap\supp_R(L)$. 
Let 
$\fa$ be an ideal of $R$ such that $\mcf\subseteq V(\fa)$, and
let $U\subseteq R\ssm \cup_{\m\in\mcf}\m$ be a multiplicatively closed set. Then for all $i\geq 0$ there are $R$-module isomorphisms
\begin{align*}
\Ext[\comp R^{\fa}]{i}{\Gamma_{\fa}(T)}{\comp R^{\fa}\otimes_{R}L}
\cong\textstyle\Ext{i}{T}{L}
&\textstyle
\cong\Ext[U^{-1}R]{i}{U^{-1}T}{U^{-1}L}.
\end{align*}
The first isomorphism is $\Comp Ra$-linear.
\end{thm}

\begin{proof}
For the first isomorphism, 
we first set $\fb=\cap_{\m\in\mcf}\m\supseteq\fa$.
Note that there
is a bijection $\mcf\to\mspec(\Comp R{b})$ given by
$\m\mapsto\m\Comp Rb$; see Fact~\ref{fact120719a}. Also, the $\m\Comp Rb$-adic completion of
$\Comp Rb$ is naturally isomorphic to $\Comp Rm$,
and we have $\Gamma_{\m\Comp Rb}(\Ga bT)=\Ga mT$.
Thus, Theorem~\ref{xlem100312a} explains the following
$\Comp Rb$-module isomorphisms
\begin{align*}
\Ext[\comp R^{\fb}]{i}{\Gamma_{\fb}(T)}{\comp R^{\fb}\otimes_{R}L}
& \textstyle
\cong\coprod_{\mathfrak{m}\in \mcf}\Ext[\comp R^{\m}]{i}{\Gamma_{\m\Comp Rb}(\Ga bT)}{\comp R^{\m}\otimes_{\comp R^{\fb}}(\comp R^{\fb}\otimes_{R}L)}\\
& \textstyle
\cong\coprod_{\mathfrak{m}\in \mcf}\Ext[\comp R^{\m}]{i}{\Gamma_\m(T)}{\comp R^{\m}\otimes_{R}L}\\
&\cong\Ext iTL.
\end{align*}
The condition $\fa\subseteq\fb$ implies that there is a natural ring homomorphism $\Comp Ra\to\Comp Rb$ that is compatible
with the maps $R\to\Comp Ra$ and $R\to \Comp Rb$.
Thus, the above isomorphisms are $\Comp Ra$-linear.
Furthermore, the same logic explains the first $\Comp Ra$-module isomorphism in the next sequence.
\begin{align*}
\Ext[\comp R^{\fa}]{i}{\Gamma_{\fa}(T)}{\comp R^{\fa}\otimes_{R}L}
& \cong\Ext[\comp R^{\fb}]{i}{\Gamma_{\fb\Comp Ra}(\Gamma_{\fa}(T))}{\comp R^{\fb}\otimes_{\Comp Ra}(\comp R^{\fa}\otimes_{R}L)}\\
& \cong\Ext[\comp R^{\fb}]{i}{\Gamma_{\fb}(\Gamma_{\fa}(T))}{\comp R^{\fb}\otimes_{R}L}\\
& \cong\Ext[\comp R^{\fb}]{i}{\Gamma_{\fb}(T)}{\comp R^{\fb}\otimes_{R}L}
\end{align*}
Combining the two sequences 
of isomorphisms, we conclude that $\Ext{i}{T}{L}\cong\Ext[\comp R^{\fa}]{i}{\Gamma_{\fa}(T)}{\comp R^{\fa}\otimes_{R}L}$.

For the second isomorphism, let $I$ be a minimal injective resolution of $L$. 
Using prime avoidance, one shows readily that $\{\m\in\supp_R(T)\mid \m\cap U=\emptyset\}=\mcf$. 
The logic of steps (1)--(3) from the proof of Theorem~\ref{xlem100312a} explains
step~(1) in the next display:
\begin{align*}
\hom{T}{I}
&\textstyle\stackrel{(1)}\cong \coprod_{\m\in\mcf}\hom{T_\m}{I}\\
&\textstyle\stackrel{(2)}\cong \hom{\coprod_{\m\in\mcf}T_\m}{I}\\
&\textstyle\stackrel{(3)}\cong\hom{\Gamma_{\fb}(T)}{I}\\
&\textstyle\stackrel{(4)}\cong\hom{\Gamma_{\fb}(T)}{\Gamma_{\fb}(I)}\\
&\textstyle\stackrel{(5)}\cong \hom{\Gamma_{\fb}(T)}{\Gamma_{\fb}(U^{-1}I)}\\
&\textstyle\stackrel{(6)}\cong\hom{\Gamma_{\fb}(T)}{U^{-1}I}\\
&\textstyle\stackrel{(7)}\cong\hom{U^{-1}T}{U^{-1}I}\\
&\textstyle\stackrel{(8)}=\hom[U^{-1}R]{U^{-1}T}{U^{-1}I}.
\end{align*}
Step (2) is standard, as $\mcf$ is finite.
Steps (3) and (7) are by Lemma~\ref{wxlem110113a}. Steps~(4) and (6) 
are from Lemma~\ref{disc120720a}\eqref{xlem100206c2}. 
Step (5) is by Lemma~\ref{xlem101120a}.
Step (8) is Lemma~\ref{lem120719b}.
Taking cohomology, one has $\Ext{i}{T}{L}\cong\Ext[U^{-1}R]{i}{U^{-1}T}{U^{-1}L}$.
\end{proof}

\begin{cor} \label{cor110517a}
Assume that $R$ is noetherian, and let $\fc$ be an intersection of finitely many maximal
ideals of $R$.
Let $T$ and $L$ be $R$-modules such that $T$ is $\fc$-torsion. Let $U\subseteq R$ be a multiplicatively closed set and let $\fa$ be an ideal of $R$. 
Then there are isomorphisms
$\Ext{i}{U^{-1}T}{L}\cong\Ext[U^{-1}R]{i}{U^{-1}T}{U^{-1}L}$ and 
$\Ext{i}{\Gamma_{\fa}(T)}{L}\cong\Ext[\comp R^{\fa}]{i}{\Gamma_{\fa}(T)}{\comp R^{\fa}\otimes_{R}L}$
for all $i$.
The first isomorphism is $U^{-1}R$-linear, and the second one is $\Comp Ra$-linear.
\end{cor}

\begin{proof}
In the next sequence, the first isomorphism is from Theoerem~\ref{xlem100312az}:
\begin{align*}
\Ext{i}{U^{-1}T}{L}
&\cong\Ext[U^{-1}R]{i}{U^{-1}(U^{-1}T)}{U^{-1}L}\cong\Ext[U^{-1}R]{i}{U^{-1}T}{U^{-1}L}.
\end{align*}
This uses the fact that $U^{-1}T$ is  $\fc$-torsion over $R$
with the equality $\supp_R(U^{-1}T)=\{\m\in\supp_R(T)\mid \m\cap U=\emptyset\}$ from 
Lemmas~\ref{lem120727a} 
and~\ref{wxlem110113a}.
These isomorphisms are $U^{-1}R$-linear by Lemma~\ref{lem120719b}.

Similarly, we have the next $\Comp Ra$-module isomorphisms
by Theoerem~\ref{xlem100312az}
\begin{align*}
\Ext{i}{\Gamma_{\fa}(T)}{L}
&\cong\Ext[\Comp Ra]{i}{\Gamma_{\fa}(\Gamma_{\fa}(T))}{\Otimes{\Comp Ra}L}\cong\Ext[\Comp Ra]{i}{\Gamma_{\fa}(T)}{\Otimes{\Comp Ra}L}
\end{align*}
using the torsionness of $\Gamma_{\fa}(T)$ and the equality $\supp_R(\Gamma_{\fa}(T))=\supp_R(T)\cap V(\fa)$ from 
Lemmas~\ref{lem120727a} 
and~\ref{lem110617b}.
\end{proof}

Our next result compares to~\cite[Theorem 4.3]{kubik:hamm1}.

\begin{thm} \label{thm110519a}
Assume that $R$ is noetherian, and let $\fc$ be an intersection of finitely many maximal
ideals of $R$.
Let $T$ be a $\fc$-torsion $R$-module, and let $M$ be a mini-max $R$-module. Let $\mathcal{F}$ be a finite subset of $\mspec(R)$ containing $\supp_R(T)\cap\supp_R(M)$, and set $\fb=\cap_{\m\in \mathcal{F}}\m$. Then for all $i$ there are $\Comp Rb$-module isomorphisms
\begin{align*}
\textstyle\Ext{i}{T}{M}
&\cong\Ext[\comp R^{\fb}]{i}{\Hom{M}{E_{\Comp Rb}}}{\md{\Ga bT}}\\
&\textstyle\cong\coprod_{\mathfrak{m}\in \mathcal{F}}
\Ext[\comp R^{\m}]{i}{\Hom{M}{\E_{R}(R/\m)}}{\md{\Ga mT}}.
\end{align*}
\end{thm}

\begin{proof} 
Lemma~\ref{lem110413b} provides an $\Comp Rb$-module isomorphism 
$(\Gamma_{\fb}(T))^{\vee(\comp R^{\fb})}\cong\md{\Gamma_{\fb}(T)}$. 
Lemma~\ref{xlem101202a} implies that $\comp R^{\fb}\otimes_R M$ is  mini-max
over $\comp R^{\fb}$.  Since $\Comp Rb$ is semi-local and complete, Fact~\ref{xpara1} shows that  $\comp R^{\fb}\otimes_R M$ is  Matlis reflexive 
over $\comp R^{\fb}$. 

Theorem~\ref{xlem100312az} provides the first 
$\Comp Rb$-module isomorphism in the next sequence:
\begin{align*}
\Ext{i}{T}{M}
&\cong\Ext[\comp R^{\fb}]{i}{\Gamma_{\fb}(T)}{\comp R^{\fb}\otimes_{R}M}\\
&\cong\Ext[\comp R^{\fb}]{i}{\Gamma_{\fb}(T)}{(\comp R^{\fb}\otimes_{R}M)^{\vee(\Comp Rb)\vee(\Comp Rb)}}\\
&\cong\Ext[\comp R^{\fb}]{i}{(\comp R^{\fb}\otimes_{R}M)^{\vee(\Comp Rb)}}{\Gamma_{\fb}(T)^{\vee(\Comp Rb)}}\\
&\cong\Ext[\comp R^{\fb}]{i}{\Hom{M}{E_{\Comp Rb}}}{\Gamma_{\fb}(T)^{\vee}}.
\end{align*}
The fact that $\comp R^{\fb}\otimes_RM$ is  Matlis reflexive over $\comp R^{\fb}$ explains the second isomorphism.
The third and fourth isomorphisms are from  adjointness.
This explains the first isomorphism in the statement of the theorem.
To verify the second  isomorphism in the statement of the theorem, argue similarly, using the isomorphism
$$\textstyle\Ext{i}{T}{M}
\textstyle
\cong\coprod_{\mathfrak{m}\in \mathcal{F}}\Ext[\comp R^{\m}]{i}{\Gamma_\m(T)}{\comp R^{\m}\otimes_{R}M}$$
from 
Theorem~\ref{xlem100312a}.
\end{proof}

\begin{disc}
The previous result shows that if $R$ is noetherian, $A$ is artinian, and $M$ is mini-max, 
then $\Ext{i}{A}{M}$ can be computed as an extension module over a complete semi-local ring with a Matlis reflexive module in the first component and a noetherian module in the second component.  Alternatively,  it can be computed as a finite coproduct of extension modules over complete local rings with Matlis reflexive modules in the first component and noetherian modules in the second component.
Specifically: 

(a) The $\comp R^{\fb}$-module $\Hom{M}{E_{\Comp Rb}}\cong (\comp R^{\fb}\otimes_RM)^{\vee(\comp R^{\fb})}$ is Matlis reflexive.
Indeed, the proof of Theorem~\ref{thm110519a} shows that  $\comp R^{\fb}\otimes_R M$ is  Matlis reflexive 
over $\comp R^{\fb}$, hence,
so is $(\comp R^{\fb}\otimes_RM)^{\vee(\comp R^{\fb})}\cong\Hom{M}{E_{\Comp Rb}}$; the isomorphism is from Hom-tensor adjointness.

(b) The $\Comp Rb$-module $\Ga bA$ is  artinian by Fact~\ref{fact120720a},
hence Matlis reflexive by Fact~\ref{xpara1} since $\Comp Rb$ is semi-local and complete.

(c) The $\comp R^{\fb}$-module
$\md{\Ga bA}\cong (\Ga bA)^{\vee(\comp R^{\fb})}$ is  noetherian (hence Matlis reflexive).
Indeed, as $\Comp Rb$ is semi-local and complete and $\Gamma_{\fb}(A)$ is artinian over $\Comp Rb$ by 
Fact~\ref{fact120720a}, the fact that $(\Gamma_{\fb}(A))^{\vee(\comp R^{\fb})}\cong\md{\Gamma_{\fb}(A)}$ 
is noetherian over $\Comp Rb$
follows from~\cite[Theorem 1.6(3)]{ooishi:mdwm}; see Lemma~\ref{lem110413b}.

Similarly, 
$\Hom{M}{\E_{R}(R/\m)}\cong (\comp R^{\fm}\otimes_RM)^{\vee(\comp R^{\fm})}$ is a Matlis reflexive $\comp R^{\fm}$-module,
$\Ga mA$ is an artinian (hence Matlis reflexive) $\Comp Rm$-module,
and $\md{\Ga mA}\cong (\Ga mA)^{\vee(\comp R^{\fm})}$ is a noetherian (hence Matlis reflexive) $\comp R^{\fm}$-module.

The following result shows, e.g., that extension functors applied to two artinian modules over arbitrary  noetherian rings can be computed as a finite 
coproduct of extension functors applied to pairs of noetherian modules over complete local rings. 
\end{disc}

\begin{cor} \label{xcor:extSwap}
Assume that $R$ is noetherian, and let $\fc$ be an intersection of finitely many
maximal ideals of $R$.
Let $T$ be a $\fc$-torsion $R$-module, and let $A$ be an artinian $R$-module.
Let $\mathcal{F}$ be a finite subset of $\mspec(R)$ containing 
$\supp_R(T)\cap\supp_R(A)$.
Setting $\fb=\cap_{\m\in \mathcal{F}}\m$, we have $\Comp Rb$-module isomorphisms
\begin{align*}
\textstyle\Ext{i}{T}{A}
&\cong\Ext[\comp R^{\fb}]{i}{{\Ga b{A}}^{\vee}}{{\Ga b{T}}^{\vee}}
\cong\coprod_{\mathfrak{m}\in \mathcal{F}}
\Ext[\comp R^{\m}]{i}{{\Ga m{A}}^{\vee}}{{\Ga m{T}}^{\vee}}\\
&\!\!\!\!\!\!\!\!\!\!\!\!\!\!\!\!\!\!\!\!\cong\Ext[\comp R^{\fb}]{i}{{\Ga b{A}}^{\vee(\Comp Rb)}}{{\Ga b{T}}^{\vee(\Comp Rb)}}
\cong\coprod_{\mathfrak{m}\in \mathcal{F}}
\Ext[\comp R^{\m}]{i}{{\Ga m{A}}^{\vee(\Comp Rm)}}{{\Ga m{T}}^{\vee(\Comp Rm)}}.
\end{align*} 
\end{cor}

\begin{proof}
The first isomorphism in the next sequence is from adjointness and is $\Comp Rb$-linear by general principles:
$$\Hom{A}{E_{\Comp Rb}}\cong (\comp R^{\fb}\otimes_RA)^{\vee(\comp R^{\fb})}\cong \Gamma_{\fb}(A)^{\vee(\comp R^{\fb})}
\cong \Gamma_{\fb}(A)^{\vee}.$$
The second isomorphism is from Fact~\ref{fact120720a},
and the third one is from Lemma~\ref{lem110413b}.
This explains the second  isomorphism in the next sequence:
\begin{align*}
\textstyle\Ext{i}{T}{A}
&\cong\Ext[\comp R^{\fb}]{i}{\Hom{A}{E_{\Comp Rb}}}{\Ga b{T}^{\vee}}
\cong\Ext[\comp R^{\fb}]{i}{{\Ga b{A}}^{\vee}}{{\Ga b{T}}^{\vee}}.
\end{align*} 
The first step is from Theorem~\ref{thm110519a}.
The other isomorphisms from the statement of the corollary are verified similarly.
\end{proof}

\section{Properties of $\Ext{i}{M}{-}$ and $\Tor{i}{M}{-}$}\label{xsec2}

This section and the next one contain the proof of Theorem~\ref{intthm120807a}.

\subsection*{Ext}
This subsection contains non-local versions of results from~\cite[Section 2]{kubik:hamm1}.

\begin{thm}\label{xthm100308a}
Assume that $R$ is noetherian.
Let $A$ and $B$ be $R$-modules such that $A$ is artinian. Let $\mathcal{F}$ be a finite subset of $\mspec(R)$ containing $\supp_R(A)\cap\supp_R(B)$. 
Let $\fa=\cap_{\m\in \mathcal{F}}\m$, and assume that $i\geq 0$ 
is such that $\mu^i_R(\m,B)$ is finite for all $\m\in\supp_R(A)\cap\supp_R(B)$. Then
$\Ext{i}{A}{B}$ is a noetherian $\comp R^{\fa}$-module. 
\end{thm}

\begin{proof}
Theorem~\ref{xlem100312a}
provides an $\Comp Ra$-module isomorphism
\begin{align*}
\textstyle\Ext{i}{A}{M}
& \textstyle
\cong\coprod_{\mathfrak{m}\in \mcf}\Ext[\comp R^{\m}]{i}{\Gamma_\m(A)}{\comp R^{\m}\otimes_{R}M}.
\end{align*}
The proof of Theorem~\ref{xlem100312a}
also shows that $\Ext[\comp R^{\m}]{i}{\Gamma_\m(A)}{\comp R^{\m}\otimes_{R}M}=0$ for all
$\m\in\mcf\ssm(\supp_R(A)\cap\supp_R(B))$.

Since the set $\mcf$ is finite, it suffices to show that 
$\Ext[\comp R^{\m}]{i}{\Gamma_\m(A)}{\comp R^{\m}\otimes_{R}M}$
is noetherian over $\Comp Rm$ for each $\m\in\mcf$.
(See the discussion of the $\Comp Ra$-module structure in the proof of
Theorem~\ref{xlem100312a}.)
From the previous paragraph, it suffices to consider 
$\m\in\supp_R(A)\cap\supp_R(B)$.
To this end, we invoke~\cite[Theorem 2.2]{kubik:hamm1}.
To apply this result, note that
Fact~\ref{fact120720a} implies that $\Ga mA$ is artinian over $\Comp Rm$,
and a straightforward computation shows that
$\mu^i_{\Comp Rm}(\Otimes{\Comp Rm}{B})=\mu^i_R(\m,B)$, which is finite.
\end{proof}

\begin{para}[Proof of Theorem~\ref{intthm120807a}\eqref{intthm120807a1}]
\label{para120807a}
Combine
Lemma~\ref{xfact100317b}
and Theorem~\ref{xthm100308a}. \qed
\end{para}

Given that so many of our previous results are for torsion modules
(not just for artinian ones) we include the following example to show that
torsionness is not enough, even in the local case.
Similar examples show the need for finiteness conditions in other similar results.

\begin{ex}\label{ex120728a}
Let $k$ be a field, and let $k^{(\mu)}$ be a $k$-vector space of infinite rank $\mu$.
Then $k^{(\mu)}$ is $\m$-torsion where $\m=0$ is the maximal ideal of $k$.
However, the module
$\Hom[k]{k^{(\mu)}}{k}\cong k^{\mu}$
is not noetherian (or artinian or mini-max) over $\comp k=k$.
\end{ex}

\begin{thm} \label{xthm:ExtMR}
Assume that $R$ is noetherian. 
Let $M$ and $M'$ be mini-max $R$-modules, and let $i\geq 0$. 
\begin{enumerate}[\rm(a)]
\item \label{xthm:ExtMR1}
If the quotient ring $R/(\ann_R{(M)}+\ann_R(M'))$ is semi-local and complete, then $\Ext{i}{M}{M'}$ is  a Matlis reflexive $R$-module.
\item \label{xthm:ExtMR2} If $R/(\ann_R{(M)}+\ann_R(M'))$ is artinian, then $\Ext{i}{M}{M'}$ has finite length.
\end{enumerate}
\end{thm}

\begin{proof}
\eqref{xthm:ExtMR1} 
Fix a noetherian submodule $N\subseteq M$ such that $M/N$ is artinian.
The containments 
$$\ann_R(M)+\ann_R(M')\subseteq\ann_R(N)+\ann_R(M')\subseteq\ann_R(\Ext{i}{N}{M'})$$ 
provide an epimorphism:
$$R/(\ann_R(M)+\ann_R(M'))\onto R/\ann_R(\Ext{i}{N}{M'}).$$
Therefore, $R/\ann_R(\Ext{i}{N}{M'})$ is semi-local and complete. Thus, Facts~\ref{xpara1}, \ref{xlem:extensions} and~\ref{xlem2of3}\eqref{xlem100430a4}
imply that $\Ext{i}{N}{M'}$ is Matlis reflexive over $R$.

Since $M/N$ is artinian, the set 
$\supp_R(M/N)\cap\supp_R(M')\subseteq\supp_R(M/N)$ is finite.
As above, the ring
$R/\ann_R(\Ext{i}{M/N}{M'})$ is semi-local and complete,
so 
Lemma~\ref{xlem100409a1}\eqref{xlem100409a1a} implies that the set 
$\mspec(R)\cap\supp_R(\Ext{i}{M/N}{M'})$ is finite.
Thus, the union
$$\mcf:=(\supp_R(M/N)\cap\supp_R(M'))\cup(\mspec(R)\cap\supp_R(\Ext{i}{M/N}{M'}))$$
is finite.
Set $\fa:=\cap_{\m\in\mcf}\m$.
Theorem~\ref{intthm120807a}\eqref{intthm120807a1} 
implies that $\Ext{i}{M/N}{M'}$ is mini-max as an $\comp R^{\fa}$-module, so it is Matlis reflexive over $R$
by Lemma~\ref{xlem100409a'}.  
Thus, Fact~\ref{xlem100614a}\eqref{xlem100614a2} implies that $\Ext{i}{M}{M'}$ is also Matlis reflexive over $R$.

\eqref{xthm:ExtMR2}  
Lemma~\ref{xlem100213b} implies that $\Ext{i}{M}{M'}$ has finite length, because of~\eqref{xthm:ExtMR1}.
\end{proof}

\begin{para}[Proof of Ext-portion of Theorem~\ref{intthm120807a}\eqref{intthm120807a3}]
\label{para120807b}
Fact~\ref{xpara1} implies that $R/\ann_R(M')$ is semi-local and complete,
hence so is $R/(\ann_R{(M)}+\ann_R(M'))$.
Theorem~\ref{xthm:ExtMR}\eqref{xthm:ExtMR1}
implies that $\Ext{i}{M}{M'}$ and $\Ext{i}{M'}{M}$ are Matlis reflexive over $R$. 
\qed
\end{para}

\begin{cor} \label{xcor:ExtMR'}
Assume that $R$ is noetherian. 
Let $M$ be a mini-max  $R$-module, and let $M'$ be a finite length  $R$-module. 
Then $\Ext{i}{M}{M'}$ and 
$\Ext{i}{M'}{M}$ have finite length over $R$ for all $i\geq 0$.
\end{cor}

\begin{proof}
Argue as in~\ref{para120807b}, using Theorem~\ref{xthm:ExtMR}\eqref{xthm:ExtMR2}.
\end{proof}

\begin{prop}\label{prop110510b}
Assume that $R$ is noetherian. 
Let $A$ be an artinian $R$-module and $M$ a mini-max $R$-module.  Let $\mathcal{F}$ be a finite subset of $\mspec(R)$ containing $\supp_R(A)\cap\supp_R(M)$,
and set $\fb=\cap_{\m\in \mathcal{F}}\m$. Then
$\Ext{i}{M}{A}$ is a Matlis reflexive $\comp R^{\fb}$-module for all $i\geq 0$.
\end{prop}

\begin{proof}
Fix a noetherian submodule $N\subseteq M$ such that $M/N$ is artinian.   Fact~\ref{xlem2of3}\eqref{xlem100430a4} implies that
$\Ext{i}{N}{A}$ is an artinian $R$-module. Since $N$ is noetherian, we have
$$\supp_R(\Ext{i}{N}{A})\subseteq\supp_R(N)\cap\supp_R(A)\subseteq\supp_R(M)\cap\supp_R(A)\subseteq\mathcal{F}.$$
Lemma~\ref{lem110517a}\eqref{item120718e} and Fact~\ref{fact120720a}
imply that $\Ext{i}{N}{A}$ is $\fb$-torsion, so Lemma~\ref{lem110413a} implies that $\Ext{i}{N}{A}$ is
an artinian $\comp R^{\fb}$-module. By Theorem~\ref{intthm120807a}\eqref{intthm120807a1} 
we have that  $\Ext{i}{M/N}{A}$ is a noetherian $\Comp Rb$-module. Since $\mcf$ is a finite set of maximal ideals, the ring $\Comp Rb$ is 
semi-local and complete. Hence, Fact~\ref{xpara1} 
implies that the $\Comp Rb$-modules $\Ext{i}{N}{A}$ and $\Ext{i}{M/N}{A}$ are Matlis reflexive.
Therefore, $\Ext{i}{M}{A}$ is a Matlis reflexive $\comp R^{\fb}$-module by 
Facts~\ref{xlem:extensions} and~\ref{xlem100614a}\eqref{xlem100614a2}.
\end{proof}

\begin{prop}\label{prop110510a}
Assume that $R$ is noetherian.
Let $M$ be a mini-max $R$-module and $N'$ a noetherian $R$-module such that  
$R/(\ann_R(M)+\ann_R(N'))$ is semi-local and complete.  Let $\mathcal{F}$ be a finite subset of $\mspec(R)$ containing  $\mspec(R)\cap V(\ann_R(M))\cap\supp_R(N')$,
and set $\fb=\cap_{\m\in \mathcal{F}}\m$. Then $\Ext{i}{M}{N'}$ is noetherian over $R$ and over $\comp R^{\fb}$ for all $i\geq 0$.
\end{prop}

\begin{proof}
Let $N$ be a noetherian submodule of $M$ such that $M/N$ is artinian. 
Because of the containment $\supp_R(M)\subseteq V(\ann_R(M))$, the fact that the quotient $R/(\ann_R(M)+\ann_R(N'))$ is  semi-local
implies that the intersection $\mspec(R)\cap\supp_R(M)\cap\supp_R(N')$ is finite.
Also, the containment $\ann_R(M)+\ann_R(N')\subseteq\ann_R(\Ext i{M}{N'})$ provides a surjection
$$R/(\ann_R(M)+\ann_R(N'))\onto R/\ann_R(\Ext i{M}{N'})$$
so we conclude that $R/\ann_R(\Ext i{M}{N'})$ is semi-local and complete.
From the containment $\ann_R(M)\subseteq\ann_R(M/N)\cap\ann_R(N)$, we also conclude that
the quotients $R/\ann_R(\Ext i{M/N}{N'})$ and $R/\ann_R(\Ext i{N}{N'})$ are semi-local and complete.

Since $M/N$ is artinian, we have $\supp_R(M/N)\subseteq\mspec(R)$, so
\begin{align*}
\mcf
&\supseteq\mspec(R)\cap V(\ann_R(M))\cap\supp_R(N')\\
&\supseteq\mspec(R)\cap\supp_R(M)\cap\supp_R(N')\\
&\supseteq\mspec(R)\cap\supp_R(M/N)\cap\supp_R(N')\\
&=\supp_R(M/N)\cap\supp_R(N').
\end{align*}
It follows by Theorem~\ref{intthm120807a}\eqref{intthm120807a1} that 
$\Ext{i}{M/N}{N'}$ is a noetherian $\comp R^{\fb}$-module. 
Furthermore, since $N'$ is noetherian, we have
\begin{align*}
\mcf
&\supseteq\mspec(R)\cap V(\ann_R(M))\cap\supp_R(N')\\
&\supseteq\mspec(R)\cap V(\ann_R(M/N))\cap V(\ann_R(N'))\\
&\supseteq\mspec(R)\cap V(\ann_R(\Ext i{M/N}{N'}))\\
&=\mspec(R)\cap \supp_R(\Ext i{M/N}{N'})
\end{align*}
where the last equality follows from Lemma~\ref{xlem100409a1}\eqref{xlem100409a1a}
since $R/\ann_R(\Ext i{M/N}{N'})$ is semi-local and complete.
Thus, Lemma~\ref{lem120721a} implies that $\Ext{i}{M/N}{N'}$ is a noetherian $R$-module. 

Since $N$ and $N'$ are noetherian over $R$, so is $\Ext iN{N'}$.
Fact~\ref{xlem100614a}\eqref{xlem100614a2} implies that
$\Ext{i}{M}{N'}$ is also  noetherian over $R$. Arguing as above, we find that
\begin{align*}
\mcf
&\supseteq\mspec(R)\cap \supp_R(\Ext i{M}{N'})
\end{align*}
so Lemma~\ref{lem120721a} implies that $\Ext{i}{M}{N'}$ is a noetherian $\Comp Rb$-module. 
\end{proof}

\subsection*{Tor}
This subsection contains non-local versions of results from~\cite[Section 3]{kubik:hamm1}.
As we see next, it is easier to work with Tor  since we can work locally.

\begin{thm}\label{xthm101104a}
Assume that $R$ is noetherian.
Let $A$ and $B$ be $R$-modules such that $A$ is artinian. 
Let $\fb\subseteq\cap_{\m\in \supp_R(A)\cap\supp_R(B)}\m$, and assume that 
$i\geq 0$ is such that 
$\beta_i^R(\m,B)$ is finite for all $\m\in\supp_R(A)\cap\supp_R(B)$. Then
$\Tor{i}{A}{B}$ is artinian over $R$ and $\fb$-torsion, hence it is an artinian $\Comp Rb$-module.
\end{thm}

\begin{proof}
To show that $\Tor{i}{A}{B}$ is artinian over $R$, we use Lemma~\ref{zlem100207a}, as follows.
As $A$ is artinian, Lemma~\ref{lem110517a}\eqref{item120718d} and Fact~\ref{fact120720a}
imply that $\supp_R(A)$ is finite.
So, the containment
$\supp_R(\Tor{i}{A}{B})\subseteq\supp_R(A)\cap\supp_R(B)$
implies that
$\supp_R(\Tor{i}{A}{B})$ is finite.
For each $\p\in\supp_R(\Tor{i}{A}{B})$, the $R_{\p}$-module $A_{\p}$ is artinian by Fact~\ref{fact120720a},
and  $\beta_i^{R_{\p}}(B_{\p})=\beta^R_i(\p,B)$
by Lemma~\ref{xfact100317b}.
Hence, the $R_{\p}$-module
$\Tor[R_{\p}]{i}{A_{\p}}{B_{\p}}\cong\Tor{i}{A}{B}_{\p}$ is artinian, by~\cite[Theorem 3.1]{kubik:hamm1}.
Thus, Lemma~\ref{zlem100207a} implies that  $\Tor{i}{A}{B}$ is artinian over $R$.

Lemma~\ref{lem110517a}\eqref{item120718e} and Fact~\ref{fact120720a}
imply that $\Tor{i}{A}{B}$ is $\fb$-torsion.
Lemma~\ref{lem110413a} implies that $\Tor{i}{A}{B}$ is an artinian $\Comp Rb$-module.
\end{proof}

One might be tempted to try to prove the previous result by applying Theorem~\ref{xthm100308a}
to $A$ and $\md B$.
When $R$ is local, this approach works. However, in the non-local case,
the fact that $\beta_i^R(\m,B)$ is finite for all $\m\in\supp_R(A)\cap\supp_R(B)$
does not necessarily imply that $\mu^i_R(\m,\md B)$ is finite for all $\m\in\supp_R(A)\cap\supp_R(\md B)$,
because the sets $\supp_R(A)\cap\supp_R(B)$ and $\supp_R(A)\cap\supp_R(\md B)$ may not be equal,
as we discuss next.

\begin{disc}\label{xlem101104a}
Assume that $R$ is noetherian. Given an $R$-module $L$, one has
$$\supp_R(L)\cap\mspec(R)\subseteq\supp_R(\md{L})\cap\mspec(R).$$
To see this, let $\m\in\supp_R(L)\cap\mspec(R)$.
Since $L_{\m}\neq 0$, there is an element $x\in L$ such that $x/1\neq 0$ in $L_{\m}$.
Thus, the submodule $L'=Rx\subseteq L$ is finitely generated and $L'_{\m}\neq 0$.
It follows that 
$$(L'^{\vee})_{\m}\cong(L'_{\m})^{\vee(R_{\m})}\neq 0.$$
The inclusion $L'\subseteq L$ yields an epimorphism
$(L^{\vee})_{\m}\onto(L'^{\vee})_{\m}\neq 0$, implying that
$(L^{\vee})_{\m}\neq 0$.
This shows that $\m\in\supp_R(\md{L})\cap\mspec(R)$, as desired.

The containment above can be strict.  
(See, however, Lemma~\ref{xlem100409a1}\eqref{xlem100409a1b}.)
If we let $R=k[X]$, $\n=RX$ and $L=\coprod_{\m\in\mspec(R)\ssm\{\n\}}R/\m$,
then the maximal ideal $\n$ is not in $\supp_R(L)$.
We claim, however, that $\n\in\supp_R(\md{L})$.
To see this, note that
$$\textstyle\md{L}\cong\prod_{\m\in\mspec(R)\ssm\{\n\}}\md{(R/\m)}\cong \prod_{\m\in\mspec(R)\ssm\{\n\}}R/\m.$$
The natural map $R\to\prod_{\m\neq\n}R/\m\cong\md L$ given by $1\mapsto\{1+\m\}$
is a monomorphism since its kernel is $\cap_{\m\neq\n}\m=0$.
It follows that $\supp_R(R)\subseteq\supp_R(\md{L})$, so $\n\in\spec(R)=\supp_R(\md{L})$.
\end{disc}

\begin{para}[Proof of Theorem~\ref{intthm120807a}\eqref{intthm120807a2}]
\label{para120807c}
Combine Lemma~\ref{xfact100317b}
and Theorem~\ref{xthm101104a}. \qed
\end{para}

\begin{thm}\label{xthm18}
Assume that $R$ is noetherian.
Let $M$ and $M'$ be mini-max $R$-modules.
Then for all $i\geq 0$, the $R$-module $\tor_i^R(M,M')$ is mini-max.
\end{thm}

\begin{proof}
Let $N$ be a noetherian submodule of $M$ such that the quotient $M/N$ is artinian. 
Fact~\ref{xlem2of3}\eqref{xlem100430a4} and Theorem~\ref{intthm120807a}\eqref{intthm120807a2} 
imply that $\Tor{i}{N}{M'}$ and  $\Tor{i}{A}{M'}$ are mini-max.  
Thus, $\Tor{i}{M}{M'}$ is mini-max by Fact~\ref{xlem100614a}\eqref{xlem100614a3}.
\end{proof}

\begin{thm} \label{xthm:torMR}
Assume that $R$ is noetherian.
Let $M$ and $M'$ be mini-max $R$-modules, and let $i\geq 0$. 
\begin{enumerate}[\rm(a)]
\item \label{xthm:torMR1}
If the quotient ring $R/(\ann_R{(M)}+\ann_R(M'))$ is  semi-local and complete then $\Tor{i}{M}{M'}$ is  a Matlis reflexive $R$-module.
\item \label{xthm:torMR2} If $R/(\ann_R{(M)}+\ann_R(M'))$ is artinian then $\Tor{i}{M}{M'}$ has finite length.
\end{enumerate}
\end{thm}

\begin{proof}
This follows from Theorem~\ref{xthm18}, using Fact~\ref{xpara1} and Lemma~\ref{xlem100213b}.
\end{proof}

\begin{para}[Proof of Tor-part of Theorem~\ref{intthm120807a}\eqref{intthm120807a3}]
\label{para120807d}
Fact~\ref{xpara1} implies that $R/\ann_R(M')$ is semi-local and complete,
hence so is $R/(\ann_R{(M)}+\ann_R(M'))$.
Thus, Theorem~\ref{xthm:torMR}\eqref{xthm:torMR1} implies that $\Tor{i}{M}{M'}$ is Matlis reflexive over $R$ for all $i\geq 0$.
\qed
\end{para}

The next  result is  proved like Corollary~\ref{xcor:ExtMR'}.

\begin{cor}\label{cor110510b}
Assume that $R$ is noetherian.
Let $M$ be a mini-max $R$-module, and let $M'$ be a finite-length $R$-module. Then $\Tor{i}{M}{M'}$ has finite length over $R$
for all $i\geq 0$.
\end{cor}

\section{Matlis Duals of Ext Modules}\label{sec110519a}

This section contains the conclusion of the proof of Theorem~\ref{intthm120807a};
see~\ref{para120807e}.
It is modeled on \cite[Section 4]{kubik:hamm1}. 
However, Lemmas~\ref{lem110405}--\ref{xlem101221a} show that the
the current work is more technically challenging than~\cite{kubik:hamm1}. 

\begin{defn}\label{xdefn100602a}
Let $L$ and $L''$ be $R$-modules, and let $J$ be an $R$-complex.
The \emph{Hom-evaluation} morphism
$$\theta_{LJL''}\colon\Otimes{L}{\Hom{J}{L''}}\to\Hom{\Hom{L}{J}}{L''}$$
is given by $\theta_{LJL''}(l\otimes\psi)(\phi)=\psi(\phi(l))$.
\end{defn}

\begin{disc}\label{xdisc100602a}
Assume that $R$ is noetherian.
Let $L$ and $L'$ be $R$-modules, and let $J$ be an injective resolution of $L'$.
Using $L''=E$ in Definition~\ref{xdefn100602a}, we have 
$\theta_{LJE}\colon\Otimes{L}{\md{J}}\to\md{\Hom{L}{J}}$.
The complex $\md{J}$ is a flat resolution of $\md{L'}$;
see, e.g., \cite[Theorem~3.2.16]{enochs:rha}.
This explains the first isomorphism in the next sequence:
\begin{align*}
\Tor{i}{L}{\md{L'}}
\xra\cong\HH_i(\Otimes{L}{\md{J}})
\xra{\HH_i(\theta_{LJE})}
&\HH_i(\md{\Hom{L}{J}})
\xra\cong\md{\Ext{i}{L}{L'}}.
\end{align*}
The second isomorphism follows from the exactness of $\md{(-)}$.
\end{disc}

\begin{defn}\label{xdefn100602b}
Assume that $R$ is noetherian.
Let $L$ and $L'$ be $R$-modules, and let $J$ be an injective resolution of $L'$.
The $R$-module homomorphism
$$\Theta^{i}_{LL'}\colon\Tor{i}{L}{\md{L'}}\to\md{\Ext{i}{L}{L'}}$$
is defined to be the composition of the the maps displayed in Remark~\ref{xdisc100602a}.
\end{defn}

\begin{disc}\label{xdisc100602b}
Assume that $R$ is noetherian.
Let $L$, $L'$, and $N$ be $R$-modules such that $N$ is noetherian.
It is straightforward to show that the map $\Theta^{i}_{LL'}$ is natural in $L$ and in $L'$.

The injectivity of $E$ implies that
$\Theta^{i}_{NL'}$ is an isomorphism;
see~\cite[Lemma~3.60]{rotman:iha}.
This explains the first of the following isomorphisms:
\begin{align*}
\md{\Ext{i}{N}{L'}}
&\cong\Tor{i}{N}{\md{L'}}&
\md{\Tor{i}{L}{L'}}
&\cong\Ext{i}{L}{\md{L'}}.
\end{align*}
The second isomorphism is a consequence of Hom-tensor adjointness.
Since Tor is commutative, the second isomorphism
implies that $\Ext{i}{L}{\md{L'}}\cong\Ext{i}{L'}{\md{L}}$.
\end{disc}

\begin{fact}\label{xfact100604a}
Assume that $R$ is noetherian.
Let $L$ and $L'$ be $R$-modules, and fix an index $i\geq 0$. Then the following diagram commutes:
$$\xymatrix@C=1.5cm{
\Ext{i}{L'}{L}
\ar[r]^-{\bidual{\Ext{i}{L'}{L}}}
\ar[d]_{\Ext{i}{L'}{\bidual L}}
& \mdd{\Ext{i}{L'}{L}}
\ar[d]^{\md{(\Theta_{L'L}^{i})}}
\\
\Ext{i}{L'}{\mdd L}
\ar[r]^-{\cong}
&\md{\Tor{i}{L'}{\md L}}.
}$$
The unlabeled isomorphism is from Remark~\ref{xdisc100602b}.
\end{fact}

\begin{lem}\label{xlem100604a}
Assume that $R$ is noetherian, and let $i\geq 0$.
\begin{enumerate}[\rm(a)]
\item \label{xlem100604a1}
If $N$ is a noetherian $R$-module and $L$ is an $R$-module, 
then the induced map $\Ext{i}{N}{\bidual L}\colon \Ext{i}{N}{L}\to\Ext{i}{N}{\mdd L}$
is an injection. 
\item \label{xlem100604a2}
Let $B$ be an $R$-module.
For each $\m\in\mspec(R)$
such that $\mu^i_R(\m,B)$ is finite, the map 
$\Ext{i}{R/\m}{\bidual B}$
is an isomorphism.
\end{enumerate}
\end{lem}

\begin{proof}
\eqref{xlem100604a1}
Remark~\ref{xdisc100602b} implies that
$\Theta^{i}_{NL}$
is an isomorphism. Hence $\md{(\Theta_{NL}^i)}$ is also an isomorphism.
The map
$\bidual{\Ext{i}{N}{L}}$
is an injection by Fact~\ref{xfact100909a}.
Using  Fact~\ref{xfact100604a} 
we conclude that $\Ext{i}{N}{\bidual L}$ is an injection.

\eqref{xlem100604a2}
Assume now that $\m\in\mspec(R)$ is such that $\mu^i_R(\m,B)$ is finite. It follows that 
$\Ext{i}{R/\m}{B}$ is a finite dimensional $R/\m$-vector space, 
so it is Matlis reflexive over $R$ by Lemma~\ref{xlem100213b}. Hence, the map $\bidual{\Ext{i}{R/\m}{B}}$
is an isomorphism. Again, using  Fact~\ref{xfact100604a}
we conclude that $\Ext{i}{R/\m}{\bidual B}$ is an isomorphism, as desired.
\end{proof}

\begin{lem} \label{lem110405}
Assume that $R$ is noetherian.
Let $B$ be an $R$-module, and assume that $\m\in\mspec(R)$ is 
such that $\mu_R^1(\m,B)$ is finite.  Then 
there is an $R$-module $B'$ and an index set $\mathcal{S}$ such that
$B\cong B'\coprod \E_R(R/\m)^{(\mathcal{S})}$ and $\mu_R^0(\m,B')$ is finite. 
\end{lem}

\begin{proof}
Set $\E(\m)=\E_R(R/\m)$, and let $\mu_R^1(\m,B)=n$. 
Note that any map $\phi\in\Hom{\E(\m)}{\E(\m)}\cong\Comp Rm$ is just multiplication by some element $r\in\Comp Rm$.
Hence, any map in $\phi\in\Hom{\E(\m)}{\E(\m)^n}\cong (\Comp Rm)^n$ is just multiplication by some vector $v\in(\Comp Rm)^n$. 
Given a vector $v\in(\Comp Rm)^n$, let $\phi_v\in\Hom{\E(\m)}{\E(\m)^n}$ denote the map that is multiplication by  $v$. 

Let $I$ be a minimal injective resolution of $B$, and decompose
$I^0=J\coprod \E(\m)^{(\mathcal{T})}=J\coprod(\coprod_{\alpha\in\mathcal{T}}\E(\m)_{\alpha})$ with $\Gamma_\m(J)=0$, where $\mathcal{T}$ is an index set. 
Here $\E(\m)_{\alpha}=\E(\m)$ for every $\alpha$; we use the subscripts to refer to specific summands. 
Let $\partial^0_I\colon I^0\to I^1$ be the first map in the injective resolution $I$. 
Then $\Gamma_\m(\partial^0_I)\colon \coprod_{\alpha\in\mathcal{T}}\E(\m)_{\alpha}\to \E(\m)^n$ 
can be described component-wise as 
$(\phi_{v_{\alpha}})_{\alpha\in \mathcal{T}}$ for  vectors $v_{\alpha}\in(\Comp Rm)^n$. 

Since $(\Comp Rm)^n$ is  noetherian over
$\Comp Rm$, so is the submodule $N:=\sum_{\alpha\in T}\Comp Rmv_{\alpha}\subseteq (\Comp Rm)^n$.  
Thus, we can choose distinct $\alpha_1,\hdots ,\alpha_m\in \mathcal{T}$ such that 
$N=\sum_{j=1}^m\Comp Rmv_{\alpha_j}$. 
Let $\mathcal{S}=\mathcal{T}\ssm\{\alpha_1,\alpha_2,\hdots ,\alpha_m\}$. 
Given $\beta\in \mathcal{S}$ choose $r_{\beta,1},\hdots ,r_{\beta,m}\in\Comp Rm$ such that
$v_{\beta}=\sum_{i=1}^m r_{\beta,i}v_{\alpha_i}$. 
For each $\beta\in\mathcal S$, set 
$$
X_{\beta}:=\left\{(e,-r_{\beta,1}e,\ldots ,-r_{\beta,m}e)\in \E(\m)_{\beta}\coprod\left(\coprod_{i=1}^m\E(\m)_{\alpha_i}\right)\mid\ e\in \E(\m)\right\}\subseteq I^0.$$ 
Then the map from $\E(\m)$ to $X_{\beta}$ defined by $e\mapsto (e,-r_{\beta,1}e,\ldots ,-r_{\beta,m}e)$ is an isomorphism. 
By construction, we have $X_{\beta}\subseteq\ker(\partial^0_I)=B$.  

Consider the submodule $X:=\sum_{\beta\in\mathcal S}X_{\beta}\subseteq B\subseteq I^0$.
It is straightforward to show that the sum defining $X$ is a direct sum.
Hence, we have $X\cong \E(\m)^{(\mathcal S)}$. In particular, $X$ is an injective submodule of $B$, so it is a summand of $B$
and a summand of $I^0$.
It is straightforward to show that $I^0\cong J\coprod X\coprod\left(\coprod_{i=1}^m\E(\m)_{\alpha_i}\right)$.
Moreover, with $B'=B\cap (J\coprod 0\coprod\left(\coprod_{i=1}^m\E(\m)_{\alpha_i}\right))$,
the module $B$ is the internal direct sum $B=B'\oplus X\cong B'\oplus \E(\m)^{(\mathcal S)}$.
Finally, since $B'\subseteq J\coprod 0\coprod\left(\coprod_{i=1}^m\E(\m)_{\alpha_i}\right)$ and $\Gamma_{\m}(J)=0$,
we conclude that $\mu^0_R(\m,B')= m$, which is finite as desired.
\end{proof}

\begin{lem}\label{xlem100604b}
Assume that $R$ is noetherian, and let $\fc$ be an intersection of finitely many 
maximal ideals of $R$.
Let $T$ and $B$ be $R$-modules such that $T$ is $\fc$-torsion.  
Assume that $i\geq 0$ 
is such that $\mu_R^i(\m,B)$ is finite for all $\m\in\supp_R(T)\cap\supp_R(B)$.
Then the induced map
$\Ext{j}{T}{\bidual{B}}\colon \Ext{j}{T}{B}\to\Ext{j}{T}{\mdd{B}}$
is an isomorphism when $j=i$, and it
is an injection when $j=i+1$.
\end{lem}

\begin{proof} 
Note that for all $\m\in \supp_R(T)\ssm\supp_R({B})$ we have $\mu^j_R(\m,{B})=0$ for all $j$, by 
Remark~\ref{xdisc101110a}.
Thus, the quantity $\mu_R^i(\m,{B})$ is finite for all $\m\in\supp_R(T)$.
As the biduality map $\bidual {B}$ is injective,  we have
an exact sequence
\begin{equation} 
\label{lem100604b1}
0\to {B}\xra{\bidual {B}}\mdd{{B}}\to\coker(\bidual {B})\to 0.
\end{equation}

Case 1:  $i=0$.
Lemma~\ref{xlem100604a} implies that for all $\m\in\supp_R(T)$ the 
induced map
$\Hom{R/\m}{\bidual {B}}$
is an isomorphism and the map $\Ext{1}{R/\m}{\bidual {B}}$ is an injection.
The long exact sequence in $\Ext{}{R/\m}{-}$ associated to~\eqref{lem100604b1} shows that
$\Hom{R/\m}{\coker(\bidual {B}})=0$ for all $\m\in\supp_R(T)$,
so $\Gamma_{\m}(\coker(\bidual {B}))=0$.
Lemmas~\ref{disc120720a}\eqref{xlem100206c2}, \ref{lem110517a}\eqref{item120718e}, 
and~\ref{lem120727a}
imply that
\begin{align*}
\Hom{T}{\coker(\bidual {B})}
&\cong\textstyle\Hom{\coprod_{\m\in\supp_R(T)}T_{\m}}{\coker(\bidual {B})}\\
&\cong \textstyle\coprod_{\m\in\supp_R(T)}\Hom{T_{\m}}{\Gamma_{\m}(\coker(\bidual {B}))}
=0.
\end{align*}
From the long exact sequence associated to $\Ext{}{T}{-}$ with respect to~\eqref{lem100604b1}, it follows that $\Hom{T}{\bidual {B}}$
is an isomorphism and $\Ext{1}{T}{\bidual {B}}$ is an injection.

Case 2: $i=1$ and $\mu^0(\m,{B}),\mu^1(\m,{B})$ are both finite for all $\m\in\supp_R(T)$.
Lemma~\ref{xlem100604a} implies that for $t=0,1$ the map
$\Ext{t}{R/\m}{\bidual {{B}}}$
is an isomorphism, and the map $\Ext{2}{R/\m}{\bidual {{B}}}$ is an injection for all $\m\in\supp_R(T)$.
From the long exact sequence associated to $\Ext{}{R/\m}{-}$ with respect to~\eqref{lem100604b1}
we conclude that for $t=0,1$ we have
$\Ext{t}{R/\m}{\coker(\bidual {{B}})}=0$ for all $\m\in\supp_R(T)$.
In other words, we have $\mu^{t}_R(\m,\coker(\bidual {{B}}))=0$ for all $\m\in\supp_R(T)$.
Let $I$ be a minimal injective resolution of $\coker(\bidual {{B}})$. Then for $t=0,1$ the module $I^{t}$ does not have $\E_R(R/\m)$ as a summand by Remark~\ref{xdisc100414a}\eqref{xdisc100414a2}. That is,
we have $\Gamma_{\m}(I^{t})=0$, so
Lemmas~\ref{disc120720a}\eqref{xlem100206c2}, \ref{lem110517a}\eqref{item120718e}, 
and~\ref{lem120727a} imply that
\begin{align*}
\Hom{T}{I^t}
&\cong\textstyle\Hom{\coprod_{\m\in\supp_R(T)}T_{\m}}{I^t}\\
&\cong \textstyle\coprod_{\m\in\supp_R(T)}\Hom{T_{\m}}{\Gamma_{\m}(I^t)}
=0.
\end{align*}
It follows that
$\Ext{t}{T}{\coker(\bidual {{B}})}=0$ for $t=0,1$. 
From the long exact sequence associated to $\Ext{}{T}{-}$
with respect to~\eqref{lem100604b1}, it follows that $\Ext{1}{T}{\bidual {{B}}}$
is an isomorphism and $\Ext{2}{T}{\bidual {{B}}}$ is an injection, as desired.

Case 3: $i=1$.
Apply Lemma~\ref{lem110405} inductively for the finitely many $\m\in\supp_R(T)$ to write
${B}\cong B'\coprod I$ where
$$\textstyle I=\coprod_{\m\in\supp_R(T)}\E_R(R/\m)^{(\mathcal{S}_\m)}$$ 
such that $\mathcal{S}_\m$ is an index set and 
$\mu^{0}(\m,B')$ is finite for all $\m\in\supp_R(T)$. 
Note that we have $\mu^1_R(\m,B')\leq\mu^1_R(\m,B)$ which is finite for all $\m\in\supp_R(T)$,
since $B'$ is a summand of $B$.
Since $I$ is injective, so is $\mdd{I}$. Hence, the maps
$\Ext{1}{T}{\bidual {I}}$ and $\Ext{2}{T}{\bidual {I}}$
are both just the map from the zero module to the zero module. Case~2 (applied to $B'$) implies that 
$\Ext{1}{T}{\bidual {B'}}$ is an isomorphism and $\Ext{2}{T}{\bidual {B'}}$ is an injection.
Since the desired result holds for $B'$ and $I$, it also holds for $B\cong B'\coprod I$.

Case 4: $i\geq 2$.  Let $J$ be a minimal injective resolution of ${B}$, and let $B''=\ker(J^{i-1}\to J^{i})$.  
As $\mu^1_R(\m,B'')=\mu^i_R(\m,B)$ is finite 
for all $\m\in\supp_R(T)$, Case~3 (applied to $B''$) shows that $\Ext{1}{T}{\bidual{B''}}$
is an isomorphism and $\Ext{2}{T}{\bidual{B''}}$ is an injection. 
A standard long exact sequence argument shows that 
$\Ext{i}{T}{\bidual {B}}$
is an isomorphism and
$\Ext{i+1}{T}{\bidual {B}}$
is an injection.
\end{proof}

\begin{lem} \label{xlem101221c}
Assume that $R$ is noetherian, and let $\fc$ be an intersection of finitely many maximal ideals of $R$.
Let  $I$, $L$, and $T$ be $R$-modules such that $T$ is $\fc$-torsion and $I$ is injective. Let $\fa\subseteq\fb:=\cap_{\m\in\supp_R(T)\cap\supp_R(I)}\m$. Then 
there are $R$-module isomorphisms
\[
T\otimes_R\Hom{I}{L}\cong T\otimes_R\Hom{\Gamma_{\fa}(I)}{L}\cong T\otimes_R\Hom{\Gamma_{\fa}(I)}{\Gamma_{\fa}(L)}.
\]
\end{lem}

\begin{proof}
Fix an isomorphism $I\cong\coprod_{\p\in\supp_R(I)}\E_R(R/\p)^{(\mu_{\p})}$. 
Set $\fb'=\cap_{\m\in\supp_R(T)}\m$,
and let $\p\in\supp_R(I)\ssm V(\fa)$.
The assumption $\fa\subseteq\fb$ implies that
$\p\notin\supp_R(T)$. Hence, using the fact that $\supp_R(T)$ is a finite set of maximal ideals,
we conclude that $\fb'\not\subseteq\p$.
Since $\E_R(R/\p)^{(\mu_{\p})}$ is an $R_\p$-module, so is $\Hom{\E_R(R/\p)^{(\mu_{\p})}}{L}$.  The condition $\fb'\nsubseteq\p$,
implies that $\fb' R_{\p}=R_{\p}$, and this explains the second step in the next display:
\begin{align*}
\textstyle\Hom{\coprod_{\p\in\supp_R(I)\ssm V(\fa)}\E_R(R/\p)^{(\mu_{\p})}}{L} \hspace{-2.2cm} \\
&\cong\textstyle\prod_{\p\in\supp_R(I)\ssm V(\fa)}\Hom{\E_R(R/\p)^{(\mu_{\p})}}{L}\\
&=\textstyle\prod_{\p\in\supp_R(I)\ssm V(\fa)}\fb'\Hom{\E_R(R/\p)^{(\mu_{\p})}}{L}\\
&=\textstyle\fb'\prod_{\p\in\supp_R(I)\ssm V(\fa)}\Hom{\E_R(R/\p)^{(\mu_{\p})}}{L}\\
&\cong\textstyle\fb'\Hom{\coprod_{\p\in\supp_R(I)\ssm V(\fa)}\E_R(R/\p)^{(\mu_{\p})}}{L}.
\end{align*}
The third step follows from the fact $\fb'$ is finitely generated, and the remaining steps are standard.   
Set $X:=\coprod_{\p\in\supp_R(I)\ssm V(\fa)}\E_R(R/\p)^{(\mu_{\p})}$,
which satisfies $\Hom X{L}=\fb'\Hom X{L}$ by the previous display. 
Lemmas~\ref{lem110517a}\eqref{item120718e} and~\ref{lem120727a}
imply that $T$ is $\fb'$-torsion, so  
$T\otimes_R\Hom{X}{L}=0$ by Lemma~\ref{xlem:tensor}. Also we have 
\[
\textstyle I\cong(\coprod_{\p\in V(\fa)\cap\supp_R(I)}\E_R(R/\p)^{(\mu_{\p})})
\coprod X\cong\Gamma_{\fa}(I)\coprod X
\]
by Fact~\ref{xfact101221a}, and it follows that
\[
\textstyle\Otimes{T}{\Hom{I}{L}}\cong\Otimes{T}{\Hom{\Gamma_{\fa}(I)\coprod X}{L}}
\cong\Otimes{T}{\Hom{\Gamma_{\fa}(I)}{L}}.
\]
This explains the first  isomorphism from the statement of the lemma, and the second one follows from 
Lemma~\ref{disc120720a}\eqref{xlem100206c2}.
\end{proof}

\begin{lem} \label{xlem101221a}
Assume that $R$ is noetherian, and let $\fc$ be an intersection of
finitely many maximal ideals of $R$.
Let $T$ and $L$ be $R$-modules such that $T$ is $\fc$-torsion. 
Let $\fa$ be an ideal contained in $\cap_{\m\in\supp_R(T)\cap\supp_R(L)}\m$.
For each index $i\geq 0$, there is an $R$-module isomorphism
\[
\Tor{i}{T}{\Hom{L}{E_{\Comp Ra}}}\cong\Tor{i}{T}{\md L}.
\]
\end{lem}

\begin{proof}
Let $I$ be a minimal injective resolution of $L$.
The minimality of $I$ implies that $\supp_R(I^j)\subseteq\supp_R(L)$ for all $j$.
Thus, Lemma \ref{xlem101221c} explains the first and third isomorphisms in the following display:
\begin{align*}
\Otimes{T}{\Hom{I}{E}}
&\cong\Otimes{T}{\Hom{\Ga{a}{I}}{\Ga{a}{E}}}\\
&\cong\Otimes{T}{\Hom{\Ga{a}{I}}{E_{\Comp Ra}}}\\
&\cong\Otimes{T}{\Hom{I}{E_{\Comp Ra}}}.
\end{align*}
The second   isomorphism is from  Lemma~\ref{xlem101021a}.
Since $E$ and $E_{\Comp Ra}$ are injective over $R$,
the complex $\Hom{I}{E}$ is a flat resolution of $\Hom{L}{E}=\md L$,
and $\Hom{I}{E_{\Comp Ra}}$ is a flat resolution of $\Hom{L}{E_{\Comp Ra}}$; see~\cite[Theorem 3.2.16]{enochs:rha}.
By taking homology in the display, we obtain the desired isomorphism.
\end{proof}

Example~\ref{ex120728a} can be used to show that it is not enough to assume that $A$ is
$\fc$-torsion in the next results.

\begin{thm}\label{xprop100601b}
Assume that $R$ is noetherian.
Let $A$ and $B$ be $R$-modules such that $A$ is artinian. 
Let $\mathcal{F}$ be a finite set of maximal ideals of $R$ containing $\supp_R(A)\cap\supp_R(B)$,
and set $\fb=\cap_{\m\in\mathcal{F}}\m$.
Assume that $i\geq 0$ is such that $\mu_R^i(\m,B)$ is finite for all 
$\m\in\supp_R(A)\cap\supp_R(B)$. Then we have the following:
\begin{enumerate}[\rm(a)]
\item\label{xprop100601b2}
There is an $R$-module isomorphism 
$\Ext{i}{A}{B}^{\vee(\Comp R\fb)}\cong\Tor{i}{A}{\md B}$.
\item\label{xprop100601b1}
If $R/(\ann_R(A)+\ann_R(B))$ is semi-local and complete, then $\Theta^{i}_{AB}$
provides an $R$-module isomorphism $\md{\Ext{i}{A}{B}}\cong\Tor{i}{A}{\md{B}}$.
\end{enumerate}
\end{thm}

\begin{proof}
\eqref{xprop100601b1}  
Assume that $R':=R/(\ann_R(A)+\ann_R(B))$ is semi-local and complete.
From the containment $\ann_R(A)+\ann_R(B)\subseteq\ann_R(\Ext iAB)$, it follows that
$R/\ann_R(\Ext iAB)$ is semi-local and complete.
Theorem~\ref{xthm100308a} implies that $\Ext{i}{A}{B}$ is noetherian over $\Comp Rb$,
so it is noetherian over
$$\Comp Rb/(\ann_R(A)+\ann_R(B))\Comp Rb
\cong \Comp{R'}b\cong \comp{R'}^{\fb R'}.$$
Since $R'$ is semi-local and complete, the ring $\comp{R'}^{\fb R'}$ is a homomorphic image of $R'$,
hence a homomorphic image of $R$. Thus, $\Ext{i}{A}{B}$ is noetherian over $R$, so Fact~\ref{xpara1}
implies that $\Ext{i}{A}{B}$ is Matlis reflexive over $R$, i.e., the 
biduality map $\bidual{\Ext{i}{A}{B}}
\colon\Ext{i}{A}{B}\to\mdd{\Ext{i}{A}{B}}$ is an isomorphism. Lemma~\ref{xlem100604b} shows that the  map
$\Ext{i}{A}{\bidual B}\colon \Ext{i}{A}{B}\to\Ext{i}{A}{\mdd B}$
is an isomorphism, so
Fact~\ref{xfact100604a} implies that
$\md{(\Theta^{i}_{AB})}$ is an isomorphism.
Since $E$ is faithfully injective, the map $\Theta^{i}_{AB}$ is also an isomorphism.

\eqref{xprop100601b2}  
We first verify that
\begin{equation}
\label{xprop100601b3}
\Tor[\Comp R\fb]{i}{\Gamma_{\fb}(A)}{(\Otimes{\Comp R\fb}{B})^{\vee(\Comp R\fb)}}\cong
\Tor{i}{A}{(\Otimes{\Comp Rb}{B})^{\vee(\Comp R\fb)}}.
\end{equation}
For this, let $P$ be a projective resolution of $A$ over $R$.
Since $\Comp R\fb$ is flat over $R$, 
the complex $\Otimes{\Comp R\fb}{P}$ is a 
projective resolution of $\Otimes{\Comp R\fb}{A}\cong \Gamma_{\fb}(A)$ over $\Comp R\fb$;
see Fact~\ref{fact120720a}. 
Using tensor-cancellation, we have
$$\Otimes[\Comp R\fb]{(\Otimes{\Comp R\fb}{P})}{(\Otimes{\Comp R\fb}{B})^{\vee(\Comp R\fb)}}\cong\Otimes{P}{(\Otimes{\Comp R\fb}{B})^{\vee(\Comp R\fb)}}$$
and 
the isomorphism~\eqref{xprop100601b3} follows by taking homology.

Set $\mcf'=\supp_R(A)\cap\supp_R(B)$ and $\fb'=\cap_{\m\in\mathcal{F}'}\m$.
We  next show that
\begin{equation}\label{eq110615a}
\Ext[\Comp R{\fb'}]{i}{\Gamma_{\fb'}(A)}{\Otimes{\Comp R{\fb'}}{B}}^{\vee(\Comp R{\fb'})}
\cong
\Tor[\Comp R{\fb'}]{i}{\Gamma_{\fb'}(A)}{(\Otimes{\Comp R{\fb'}}{B})^{\vee(\Comp R{\fb'})}}.
\end{equation}
Since $\mcf'$ is a finite set of maximal ideals, the ring $\Comp R{b'}$ is semi-local and complete.
Fact~\ref{fact120720a} implies that $\Gamma_{\fb'}(A)$ is artinian over  $\comp R^{\fb'}$.
The maximal ideals of $\Comp R{b'}$ are of the form $\m\Comp R{b'}$ with $\m\in\mcf'$;
see Fact~\ref{fact120719a}.
For each such $\m$, the quantity 
$\mu_{\Comp R{b'}}^i(\m\Comp R{b'},\Otimes{\Comp R{b'}}{B})=\mu_R^i(\m,B)$ is finite,
so the isomorphism~\eqref{eq110615a} follows from part~\eqref{xprop100601b1}.

Note that Theorem~\ref{xthm100308a} implies that $\Ext iAB$ is an $\Comp Rb$-module and an $\Comp R{b'}$-module.
Theorem~\ref{xlem100312az} explains the first isomorphism in the next display:
\begin{align*}
\Ext{i}{A}{B}^{\vee(\Comp R{\fb'})}
&\cong
\Ext[\Comp R{\fb'}]{i}{\Gamma_{\fb'}(A)}{\Otimes{\Comp R{\fb'}}{B}}^{\vee(\Comp R{\fb'})}\\
&\cong
\Tor[\Comp R{\fb'}]{i}{\Gamma_{{\fb'}}(A)}{(\Otimes{\Comp R{\fb'}}{B})^{\vee(\Comp R{\fb'})}}\\
&\cong
\Tor{i}{A}{(\Otimes{\Comp R{b'}}{B})^{\vee(\Comp R{\fb'})}}\\
&\cong
\Tor{i}{A}{\Hom{B}{E_{\Comp R{b'}}}}\\
&\cong\Tor{i}{A}{\md B}.
\end{align*}
The second step is  from~\eqref{eq110615a}. 
The third  step is from~\eqref{xprop100601b3}, in the special case where 
$\mcf=\mcf'$.
The fourth step is from Hom-tensor adjointness.
The fifth step is from Lemma~\ref{xlem101221a}.

To complete the proof, recall that $\Comp Rb\cong\prod_{\m\in\mcf}\Comp Rm$
and $\Comp R{b'}\cong\prod_{\m\in\mcf'}\Comp Rm$.
It follows that $\Comp R{b'}\cong\Comp Rb/\fa$ where $\fa$ is an idempotent ideal of $\Comp Rb$.
Since $\fa$ is idempotent, we have $(\Comp Rb)^{\wedge\fa}\cong\Comp Rb/\fa\cong\Comp R{b'}$.
As $\Ext iAB$ is an $\Comp R{b'}$-module, it is $\fa$-torsion, so Lemma~\ref{lem110413b} provides the first isomorphism
in the next sequence
$$
\Ext{i}{A}{B}^{\vee(\Comp R{\fb})}\cong \Ext{i}{A}{B}^{\vee(\Comp R{\fb'})}\cong\Tor{i}{A}{\md B}.
$$
The second isomorphism is from the previous display.
\end{proof}

\begin{disc}\label{disc120820a}
Lemma~\ref{xlem100604b} and Theorem~\ref{xprop100601b} answer~\cite[Question 4.8]{kubik:hamm1}.
\end{disc}

\begin{cor}\label{xcor100602a}
Assume that $R$ is noetherian.
Let $A$ and $M$ be $R$-modules such that $A$ is artinian and $M$ is mini-max.  Let $\mathcal{F}$ be a finite set of maximal ideals of $R$ containing 
$\supp_R(A)\cap\supp_R(M)$, and set $\fb=\cap_{\m\in\mathcal{F}}\m$.
For each index $i\geq 0$, one has 
an $R$-module isomorphism
$\Ext{i}{A}{M}^{\vee(\Comp R\fb)}\cong\Tor{i}{A}{\md M}$.
\end{cor}

\begin{proof}
Combine Lemma~\ref{xfact100317b}
and Theorem~\ref{xprop100601b}\eqref{xprop100601b2}.
\end{proof}

\begin{thm}\label{xprop100601a}
Assume that $R$ is noetherian.
Let $M$ and $B$ be $R$-modules such that $M$ is mini-max and the quotient 
$R/(\ann_R(M)+\ann_R(B))$ is semi-local and complete.  
Assume that $i\geq 0$ is such that 
$\mu_R^i(\m,B)$ and $\mu_R^{i+1}(\m,B)$ are finite for all 
$\m\in\supp_R(M)\cap\supp_R(B)\cap\mspec(R)$. Then
$\Theta^{i}_{MB}$
is an isomorphism, so 
$$\md{\Ext{i}{M}{B}}\cong\Tor{i}{M}{B^\vee}.$$
\end{thm}

\begin{proof}
Since $M$ is mini-max over $R$,
there is an exact sequence of $R$-modules homomorphisms $0\to N\to M\to A\to 0$ 
such that $N$ is noetherian and $A$ is artinian.  The long 
exact sequences associated to $\Tor{}{-}{\md{B}}$ and $\md{\Ext{}{-}{B}}$ 
fit into the following commutative diagram:
$$\xymatrix{
\cdots\ar[r]
&\Tor{i}{N}{B^\vee}\ar[r]\ar[d]^{\Theta^{i}_{NB}}
&\Tor{i}{M}{B^\vee}\ar[r]\ar[d]^{\Theta^{i}_{MB}}
&\Tor{i}{A}{B^\vee}\ar[d]^{\Theta^{i}_{AB}}\ar[r]
&\cdots\\
\cdots\ar[r]
&\Ext{i}{N}{B}^\vee\ar[r]
&\Ext{i}{M}{B}^\vee\ar[r]
&\Ext{i}{A}{B}^\vee\ar[r]
&\cdots.
}$$
By Remark \ref{xdisc100602b}, the maps $\Theta^{i}_{NB}$ and $\Theta^{i-1}_{NB}$ are  isomorphisms.
Theorem~\ref{xprop100601b}\eqref{xprop100601b1} implies that 
$\Theta^{i}_{AB}$ and $\Theta^{i+1}_{AB}$ are isomorphisms.  
Hence, the map $\Theta^{i}_{MB}$ is an isomorphism by the Five Lemma.
\end{proof}

\begin{cor} \label{cor120729a}
Assume that $R$ is noetherian.
Let $M$ and $B$ be $R$-modules such that $M$ is Matlis reflexive.
Assume that $i\geq 0$ is such that 
$\mu_R^{i}(\m,B)$ and $\mu_R^{i+1}(\m, B)$ are finite for all 
$\m\in\supp_R(M)\cap\supp_R(B)\cap\mspec(R)$. Then
$\Theta^{i}_{MB}$
is an isomorphism, so 
$\md{\Ext{i}{M}{B}}\cong\Tor{i}{M}{B^\vee}$.
\end{cor}

\begin{proof}
Combine Fact~\ref{xpara1}
and Theorem~\ref{xprop100601a}.
\end{proof}

\begin{para}[Proof  of Theorem~\ref{intthm120807a}\eqref{intthm120807a4}]
\label{para120807e}
Apply Fact~\ref{xpara1}, Lemma~\ref{xfact100317b},
and Theorem~\ref{xprop100601a}.
\qed
\end{para}

\begin{cor}\label{xcor110213a}
Assume that $R$ is noetherian.
Let $M$ and $M'$ be mini-max $R$-modules such that the quotient 
$R/(\ann_R(M)+\ann_R(M'))$ is semi-local and complete.  
Let $\mathcal{F}$ be a finite set of maximal ideals of $R$ containing $V(\ann_R(M))\cap V(\ann_R(M'))\cap\mspec(R)$,
and set $\fb=\cap_{\m\in\mathcal{F}}\m$.
Then for all $i\geq 0$ the map
$\Theta^{i}_{MM'}$
is an isomorphism, so 
$$\Ext{i}{M}{M'}^{\vee(\Comp R\fb)}\cong\md{\Ext{i}{M}{M'}}\cong\Tor{i}{M}{M'^\vee}.$$
\end{cor}

\begin{proof}
Combine Lemma~\ref{lem110413b} and Theorem~\ref{xprop100601a}.
\end{proof}

\section{Length and Vanishing of $\Hom[]{L}{L'}$ and $L\otimes L'$}
\label{sec120807b}

This section includes the proof of Theorem~\ref{intthm120807b}
as well as vanishing results for Ext and Tor, including a description of the associated primes of certain Hom-modules.
Most of the results of this section do not assume that $R$ is noetherian.
Note that, in the next result,  the integers $t$ and $\alpha_\m$ exist, say, when $T$
or $T'$ is artinian.

\begin{lem} \label{xcor28}
Let $\fa$ and $\fa'$ be   intersections of finitely many maximal ideals of $R$. Let $T$ be an $\fa$-torsion
$R$-module, and let $T'$ be an $\fa'$-torsion $R$-module. 
Let $\mathcal{F}$ be a subset of $\mspec(R)$ containing $\supp_R(T)\cap\supp_R(T')$, and
let $\fb$ be an ideal contained in $\cap_{\m\in\mathcal{F}}\m$. 
\begin{enumerate}[\rm(a)]
\item\label{xcor28a}
Then there is a $\Comp Rb$-module isomorphism
$T\otimes_RT'
\textstyle\cong\coprod_{\m\in\mathcal{F}}T_\m\otimes_R T'_\m$.
\item\label{xcor28b}
Assume that for each $\m\in\mathcal{F}$ there is an integer $\alpha_{\m}\geq 0$ such that 
either $\mathfrak{m}^{\alpha_{\m}}T=\mathfrak{m}^{\alpha_{\m}+1}T$
or $\mathfrak{m}^{\alpha_{\m}}T'=\mathfrak{m}^{\alpha_{\m}+1}T'$.
Then there exists a $\Comp Rb$-module isomorphism
$T\otimes_RT'
\textstyle\cong\coprod_{\m\in\mathcal{F}}(T/\m^{\alpha_{\m}}T)
\otimes_R(T'/\m^{\alpha_{\m}}T')$.
\item\label{xcor28c}
Assume that there is an integer $t\geq 0$ such that 
$\mathfrak{b}^tT=\mathfrak{b}^{t+1}T$ or $\mathfrak{b}^tT'=\mathfrak{b}^{t+1}T'$. 
Then there is a $\Comp Rb$-module isomorphism
$T\otimes_RT'
\textstyle\cong (T/\fb^tT)\otimes_R (T'/\fb^tT')$.
\end{enumerate}
\end{lem}

\begin{proof} 
\eqref{xcor28a}
In the following sequence, the first step is from Lemmas~\ref{lem110517a}\eqref{item120718e} and~\ref{lem120727a}:
\begin{align*}
T\otimes_R T'
&
\textstyle \cong\coprod_{\m\in\supp_R(T)}\Otimes[R]{T_{\m}}{T'}
\textstyle \cong\coprod_{\m\in\supp_R(T)}\Otimes[R]{T_{\m}}{T'_{\m}}
\cong\coprod_{\m\in\mathcal{F}}T_\m\otimes_R T'_\m.
\end{align*}
The remaining steps are standard, using the condition
$\mcf\supseteq\supp_R(T)\cap\supp_R(T')$.
Since $T$ is $\fa$-torsion and $\fa$ is a finite intersection of maximal ideals, it follows that $T_{\m}$ is $\m R_{\m}$-torsion for all $\m\in\mspec(R)$, and similarly for $T'_{\m}$.
In particular, for all $\m\in\mcf$, the modules $T_{\m}$ and $T'_{\m}$ are $\fb R_{\m}$-torsion, since $\fb R_{\m}\subseteq\m R_{\m}$, hence $\fb$-torsion. It follows that the modules in the
previous display are $\fb$-torsion. Thus, Lemma~\ref{disc120720a}\eqref{xlem100206c1}
implies that the $R$-module isomorphisms are $\Comp Rb$-linear.

\eqref{xcor28b}
If $\m^{\alpha_\m}T=\m^{\alpha_\m+1}T$, then $\m^{\alpha_\m}T_\m=\m^{\alpha_\m+1}T_\m$; since we have 
$T/\m^{\alpha_\m}T\cong T_\m/\m^{\alpha_\m}T_\m$, in this case Lemma~\ref{xlem:tensor} provides an isomorphism 
$$T_\m\otimes_R T'_\m\cong(T/\m^{\alpha_{\m}}T)
\otimes_R(T'/\m^{\alpha_{\m}}T').$$
Similarly, the same isomorphism holds if $\m^{\alpha_\m}T'=\m^{\alpha_\m+1}T'$, 
and the isomorphism
$\coprod_{\m\in\mcf}\Otimes[R]{T_{\m}}{T'_{\m}}\cong\coprod_{\m\in\mathcal{F}}(T/\m^{\alpha_{\m}}T)
\otimes_R(T'/\m^{\alpha_{\m}}T')$ follows.
This isomorphism is $\Comp Rb$-linear as in part~\eqref{xcor28a}.

\eqref{xcor28c}
If $\fb^t T=\fb^{t+1} T$, then $\fb^t T_{\m}=\fb^{t+1}T_{\m}$, and
Lemma~\ref{xlem:tensor} shows that 
$$T_\m\otimes_R T'_\m\cong(T_\m/\fb^tT_\m)
\otimes_R(T'_\m/\fb^tT'_\m)$$
for all $\m\in\mcf$. This explains the second step in the next display:
\begin{align*}
\textstyle (T/\fb^tT)\otimes_R (T'/\fb^t T')
&\textstyle \cong\textstyle\coprod_{\m\in\mcf}(T_\m/\fb^tT_\m)\otimes_R (T'_{\m}/\fb^t T'_{\m})\\
&\textstyle\cong\coprod_{\m\in\mathcal{F}}T_\m\otimes_R T'_\m\\
&\cong T\otimes_RT'.
\end{align*}
The other steps follow from part~\eqref{xcor28a}.
These isomorphisms are $\Comp Rb$-linear as in part~\eqref{xcor28a}.
The same isomorphisms hold by symmetry if $\fb^t T'=\fb^{t+1} T'$.
\end{proof}

The next result is proved like Lemma~\ref{xcor28}\eqref{xcor28a}.
For the sake of brevity, we leave similar versions of Lemma~\ref{xcor28}\eqref{xcor28b}--\eqref{xcor28c}
for the interested reader.

\begin{lem} \label{xcor28'}
Let $\fa$  be an  intersection of finitely many maximal ideals of $R$. Let $T$
and $L$ be 
$R$-modules such that $T$ is  $\fa$-torsion.
Let $\mathcal{F}$ be a subset of $\mspec(R)$ such that $\mcf\supseteq\supp_R(T)\cap\supp_R(L)$, and
let $\fb$ be an ideal contained in $\cap_{\m\in\mathcal{F}}\m$. 
Then there is a $\Comp Rb$-module isomorphism
$T\otimes_RL
\textstyle\cong\coprod_{\m\in\mathcal{F}}T_\m\otimes_R L_\m$.
\end{lem}

\begin{prop} \label{xprop:tensorLength}
Let $\fa$ and $\fa'$ be  finite intersections of maximal ideals of $R$. Let $T$ be an $\fa$-torsion
$R$-module, and let $T'$ be an $\fa'$-torsion $R$-module. 
Let $\mathcal{F}$ be a subset of $\mspec(R)$ containing $\supp_R(T)\cap\supp_R(T')$, and
let $\fb$ be an ideal contained in $\cap_{\m\in\mathcal{F}}\m$. 
Assume that there is an integer $t\geq 0$ such that 
$\mathfrak{b}^tT=\mathfrak{b}^{t+1}T$. 
Assume that for each $\m\in\mathcal{F}$ there is an integer $\alpha_{\m}\geq 0$ such that 
$\mathfrak{m}^{\alpha_{\m}}T=\mathfrak{m}^{\alpha_{\m}+1}T$.
Then there are inequalities
\begin{align*}
\len_R(T\otimes_RT')\hspace{-1.2cm}\\
&\textstyle\leq \sum_{\m\in\mathcal{F}}\min\{\len_{R}(T/\m^{\alpha_{\m}}T)\len_{R}(T'/\m T'), 
\len_{R}(T/\m T)\len_{R}(T'/\m^{\alpha_{\m}}T')\} \\
&\leq  \len_R\left(T/\fb^tT\right)\max\{\len_{R}(T'/\m T')\mid \m\in\mathcal{F}\}\\
&\leq  \len_R\left(T/\fb^tT\right)\len_R(T'/\fb T').
\end{align*}
Here we use the convention $0\cdot\infty=0$.
\end{prop}

\begin{proof}
Note that for all $\m\in\mspec(R)$ and all $n\geq 0$ we have  
$\len_{R_{\m}}(T_{\m}/\m^{n}T_{\m})=\len_{R}(T/\m^{n}T)$
and $\len_{R_{\m}}(T'_{\m}/\m^{n}T'_{\m})=\len_{R}(T'/\m^{n}T')$.
Thus, the proof of~\cite[Theorem~3.8]{kubik:hamm1} shows that for each $\m\in\mcf$ one has
$$\len_{R}(T_{\m}\otimes_{R}T'_{\m}) 
\!\!\leq\!\min\{\len_{R}(T/\m^{\alpha_{\m}}T)\len_{R}(T'/\m T'),
\len_{R}(T/\m T)\len_{R}(T'/\m^{\alpha_{\m}}T')\}$$
and this explains  step (2) in the next display: 
\begin{align*}
\len_R(T\otimes_RT')\hspace{-1.3cm}\\
&\textstyle\stackrel{(1)}{=}\sum_{\m\in\mathcal{F}}\len_{R}(T_{\m}\otimes_{R}T'_{\m})\\
&\textstyle\stackrel{(2)}{\leq}\sum_{\m\in\mathcal{F}}\min\{\len_{R}(T/\m^{\alpha_{\m}}T)\len_{R}(T'/\m T'),\len_{R}(T/\m T)\len_{R}(T'/\m^{\alpha_{\m}}T')\} \\
&\textstyle\stackrel{(3)}{\leq}\sum_{\m\in\mathcal{F}}\len_{R}(T/\m^{\alpha_{\m}}T)\len_{R}(T'/\m T')\\
&\textstyle\stackrel{(4)}{\leq}\left(\sum_{\m\in\mathcal{F}}\len_{R}(T/\m^{\alpha_{\m}}T)\right)\left(\max\{\len_{R}(T'/\m T')\mid \m\in\mathcal{F}\}\right)\\
&\textstyle\stackrel{(5)}{\leq}\len_R(T/\fb^tT)\max\{\len_{R}(T'/\m T')\mid \m\in\mathcal{F}\})\\
&\textstyle\stackrel{(6)}{\leq}\len_R(T/\fb^tT)\len_{R}(T'/\fb T').
\end{align*}
Step (1)
follows from 
Lemma~\ref{xcor28}\eqref{xcor28a}, and steps (3)--(4) are routine.

For step (5), since $\fb^tT=\fb^{t+1}T$, it follows that $\fb^tT_\m=\fb^{t+1}T_\m$ for all $\m\in\mspec(R)$. We conclude that 
$\fb^t T_\m=\fb^{t+\alpha_\m}T_\m\subseteq \m^{t+\alpha_\m}T_\m=\m^{\alpha_\m}T_\m$
for all $\m\in\mcf$. This explains  step (8) in the next display:
\begin{align*}
\len_R(T/\fb^tT)
&\textstyle\stackrel{(7)}{=}\sum_{\m\in\mathcal{F}}\len_R(T_\m/\fb^tT_\m)\\
&\stackrel{(8)}{\geq}\textstyle\sum_{\m\in\mathcal{F}}\len_R(T_\m/\m^{\alpha_\m}T_\m)\\
&\textstyle\stackrel{(9)}{=}\sum_{\m\in\mathcal{F}}\len_R(T/\m^{\alpha_\m}T)
\end{align*}
Step (7) follows from Lemma~\ref{xcor28'} applied to the tensor product $\Otimes[R]{T}{(R/\fb^t)}$,
and step (9) is standard. This explains step (5).

Since $\fb\subseteq\m$ for each $\m\in\mcf$, we have an epimorphism $T'/\fb T'\onto T'/\m T'$.
This explains step (6), and completes the proof.
\end{proof}

\begin{cor}\label{para120807g}
If $A$ and $A'$ are artinian $R$-modules, then $\Otimes{A}{A'}$ has finite length.
\end{cor}

\begin{proof}
Lemma~\ref{xlem100213b}
implies that the quantities
$\len_R(A/\m^\alpha A)$ and $\len_R(A'/\m^\alpha A')$ are finite for all $\m\in\mspec(R)$
and all $\alpha\geq 1$.
Thus, the finiteness of $\len_R(A\otimes_RA')$ follows from Proposition~\ref{xprop:tensorLength}.
\end{proof}

The next  result also applies, e.g., when $T$ and $T'$ are artinian.

\begin{prop}\label{xprop1}
Let $\fa$ and $\fa'$ be  finite intersections of maximal ideals of $R$. Let $T$ be an $\fa$-torsion
$R$-module, and let $T'$ be an $\fa'$-torsion $R$-module. 
Set $\fb=\cap_{\m\in\mathcal{F}}\m$, where
$\mathcal{F}$ is a finite subset of $\mspec(R)$ containing $\supp_R(T)\cap\supp_R(T')$.
Then the following conditions are equivalent:
\begin{enumerate}[\rm(i)]
\item\label{xprop1a} $T\otimes_RT'=0$;
\item\label{xprop1b} $\supp_R(T/\fb T)\cap\supp_R(T'/\fb T')=\emptyset$; 
\item\label{xprop1b'} 
For all $\m\in\mathcal{F}$,
either $T=\mathfrak{m}T$ or $T'=\mathfrak{m}T'$; and
\item\label{xprop1b''} 
For all $\m\in\mspec(R)$,
either $T=\mathfrak{m}T$ or $T'=\mathfrak{m}T'$.
\end{enumerate} 
\end{prop}

\begin{proof}
The implication
\eqref{xprop1b''}$\implies$\eqref{xprop1b'}
is trivial since $\mcf\subseteq\mspec(R)$.

\eqref{xprop1a}$\implies$\eqref{xprop1b''}:   
Assume that $T\otimes_R T'=0$. For each $\m\in\mspec(R)$, we have
\begin{align*}
0&= R/\m\otimes_R(T\otimes_RT')\\
&\cong (R/\m\otimes_RT)\otimes_{R/\m}(R/\m\otimes_RT')\\
&\cong (T/\m T)\otimes_{R/\m}(T'/\m T').
\end{align*}
The isomorphisms are standard.
Since $T/\m T$ and $T'/\m T'$ are vector spaces over $R/\m$,
it follows that either $T/\m T=0$ or $T'/\m T'=0$, as desired.

\eqref{xprop1b'}$\implies$\eqref{xprop1a} and \eqref{xprop1b'}$\implies$\eqref{xprop1b}:   
Assume that for each $\m\in\mcf$,
either $T=\mathfrak{m}T$ or $T'=\mathfrak{m}T'$.
Then  Lemma~\ref{xcor28}\eqref{xcor28b} implies that
\[
\textstyle T\otimes_RT'
\cong\coprod_{\m\in\mathcal{F}}(T/\m^{0}T)\otimes_R(T'/\m^{0}T')=0.
\]
For each $\m\in\mathcal{F}$ we have $\fb R_{\m}=\m R_{\m}$.
If $T=\m T$, then
this implies that $(T/\fb T)_{\m}=T_{\m}/\fb T_{\m}=T_{\m}/\m T_{\m}=0$, so $\m\not\in\supp_R(T/\fb T)$. 
Similarly, if  $T'=\m T'$, then
$\m\not\in\supp_R(T'/\fb T')$. 
This explains the third step in the next display:
\begin{align*}
\supp_R(T/\fb T)\cap\supp_R(T'/\fb T')\hspace{-2cm}\\
&\subseteq \supp_R(T)\cap\supp_R(T')\\
&\subseteq\mcf\\
&\subseteq(\spec(R)\ssm\supp_R(T/\fb T))\cup(\spec(R)\ssm\supp_R(T'/\fb T'))\\
&=\spec(R)\ssm(\supp_R(T/\fb T)\cap\supp_R(T'/\fb T')).
\end{align*}
The other steps are routine. It follows that the set $\supp_R(T/\fb T)\cap\supp_R(T'/\fb T')$ is contained in its own
compliment, so it must be empty.

\eqref{xprop1b}$\implies$\eqref{xprop1b'}: Assume that $\supp_R(T/\fb T)\cap\supp_R(T'/\fb T')=\emptyset$.
Let $\m\in\mathcal{F}$. Without loss of generality assume $\m\notin\supp_R(T/\fb T)$. Therefore, 
we have $0=(T/\fb T)_{\m}=T_{\m}/\fb T_{\m}=T_{\m}/\m T_{\m}$; hence $T_{\m}=\m T_{\m}$. Since 
$T\cong\coprod_{\n\in\supp_R(T)}T_{\n}$ and $T_{\n}=\m T_{\n}$ for all maximal ideals $\n\neq \m$ it follows that $T=\m T$.
\end{proof}

\begin{prop} \label{xlem:homSwap}
Let $\fc$ be an intersection of finitely many maximal ideals of $R$. 
Let $L$ and $T$ be $R$-modules such that $T$ is $\fc$-torsion, and 
let $\mathcal{F}$ be a subset of $\mspec(R)$ containing
$\supp_R(T)\cap\ass_R(L)$. For each ideal  $\fa\subseteq\cap_{\m\in\mathcal{F}}\m$, one has
\[
\textstyle\Hom{T}{L}\cong\Hom{\Gamma_{\fa}(T)}{\Gamma_{\fa}(L)}
\cong\coprod_{\m\in\mathcal{F}}\Hom{\Ga mT}{\Gamma_{\m}(L)}.
\]
\end{prop}

\begin{proof}
The first step in the next display follows from Lemmas~\ref{lem110517a}\eqref{item120718e} 
and~\ref{lem120727a}:
\begin{align*}
\Hom[R]{T}{L}&\textstyle\cong\coprod_{\m\in\supp_R(T)}\Hom{\Ga mT}{L}\\
&\textstyle\cong\coprod_{\m\in\supp_R(T)}\Hom[R]{\Ga mT}{\Gamma_{\m}(L)}\\
&\textstyle\cong\coprod_{\m\in\mathcal{F}}\Hom[R]{\Ga mT}{\Gamma_{\m}(L)}.
\end{align*}
The second step is from Lemma~\ref{disc120720a}\eqref{xlem100206c2}.
The third step  follows from the fact that for all maximal ideals $\m\notin\mathcal{F}$ either $T_{\m}\cong\Ga mT=0$ or $\Gamma_\m(L)=0$; 
see Lemma~\ref{lem110517a}\eqref{item120718d}.

Since we have $\mcf\supseteq\supp_R(T)\cap\ass_R(L)\supseteq\supp_R(\Ga aT)\cap\ass_R(L)$, 
the first paragraph of this proof gives the second step in the next sequence:
\begin{align*}
\Hom[R]{\Ga aT}{\Ga aL}
&\cong\Hom[R]{\Ga aT}{L}\\
&\textstyle\cong\coprod_{\m\in\mcf}\Hom{\Ga m{\Ga aT}}{\Ga m{L}}\\
&\textstyle\cong\coprod_{\m\in\mathcal{F}}\Hom[R]{\Ga mT}{\Gamma_{\m}(L)}.
\end{align*}
The first step is from Lemma~\ref{disc120720a}\eqref{xlem100206c2}.
For the third step, note that each $\m\in\mcf$ satisfies $\fa\subseteq\m$, so we have
$\Ga m{\Ga aT}=\Ga {m+a}T=\Ga mT$.
\end{proof}

In the next result, the assumption  ``$\mu^0_R(\m,B)$ is finite for all $\m\in V(\fa)$'' 
is equivalent to the condition $\len_R(0:_B\fa)<\infty$.

\begin{prop} \label{xlem:homSwap2} 
Assume that $R$ is noetherian,
and let $\fc$ be an intersection of finitely many maximal ideals of $R$.
Let $T$ and $B$ be $R$-modules such that $T$ is $\fc$-torsion,
and let $\mathcal{F}$ be a finite subset of $\mspec(R)$ containing $\supp_R(T)\cap\ass_R(B)$. 
Set $\fa=\cap_{\m\in\mathcal{F}}\m$, and
assume that $\mu^0_R(\m,B)$ is finite for all $\m\in \mcf$. Then we have
\[
\textstyle\Hom{T}{B}\cong\Hom[\comp R^{\fa}]{\Gamma_{\fa}(B)^{\vee}}{\Gamma_\fa(T)^{\vee}}
\cong\coprod_{\m\in\mathcal{F}}\Hom[\comp R^{\m}]{\Gamma_{\m}(B)^{\vee}}{\Ga mT^{\vee}}.
\]
\end{prop}

\begin{proof}
Since $\mu^0_R(\m,B)$ is finite for all $\m\in \mcf$, we know that 
$$\textstyle
\Ga a{\E_R(B)}\cong\coprod_{\m\in \mcf}\E_R(R/\m)^{\mu^0_R(\m,B)}$$ 
is
an artinian $R$-module containing $\Ga aB$. It follows that $\Ga aB$ is artinian over $R$ with $\supp_R(\Ga aB)\subseteq \mcf$.
Since $T$ is $\fc$-torsion, Lemmas~\ref{lem120727a}
and~\ref{lem120719a}\eqref{lem120719a1} imply that 
$\supp_R(\Ga aT)\subseteq V(\fa)=\mcf$,
so Corollary~\ref{xcor:extSwap} explains the second step in the next sequence:
\begin{align*}
\textstyle\Hom{T}{B}
&\textstyle\cong\Hom{\Ga aT}{\Ga aB}\\
&\textstyle\cong\Hom[\comp R^{\fa}]{\Gamma_{\fa}(\Ga aB)^{\vee}}{\Gamma_\fa(\Ga aT)^{\vee}}\\
&\textstyle=\Hom[\comp R^{\fa}]{\Gamma_{\fa}(B)^{\vee}}{\Gamma_\fa(T)^{\vee}}.
\end{align*}
The first step is from Proposition~\ref{xlem:homSwap}.

By construction, we have $\mcf\supseteq\supp_R(\Ga aT)\cap\supp_R(\Ga aB)$, so another application of Corollary~\ref{xcor:extSwap} and
Proposition~\ref{xlem:homSwap}   explains the first and second steps in the next sequence:
\begin{align*}
\textstyle\Hom{T}{B}
&\textstyle\cong\Hom{\Ga aT}{\Ga aB}\\
&\textstyle\cong\coprod_{\m\in\mathcal{F}}\Hom[\comp R^{\m}]{\Gamma_{\m}(\Ga aB)^{\vee}}{\Ga m{\Ga aT}^{\vee}}
\\
&\textstyle\cong\coprod_{\m\in\mathcal{F}}\Hom[\comp R^{\m}]{\Gamma_{\m}(B)^{\vee}}{\Ga mT^{\vee}}.
\end{align*}
The third step follows from the fact that every $\m\in\mcf$ satisfies $\m\supseteq\fa$.
\end{proof}

\begin{disc}\label{disc120729a}
In the previous result,
note that 
$\Gamma_{\fa}(B)^{\vee}$ is a noetherian $\comp R^{\fa}$-module while
$\Gamma_{\m}(B)^{\vee}$ is a noetherian $\comp R^{\m}$-module. 
Indeed, since $\Ga aB$ is artinian over $R$ and $\fa$-torsion, Lemma~\ref{lem110413a} implies that
$\Ga aB$ is artinian over $\Comp Ra$.
As the ring $\Comp Ra$ is semi-local and complete, Lemma~\ref{lem110413b} and Theorem~\ref{xthm100308a} imply that
$\md{\Ga aB}\cong\Ga aB^{\vee(\Comp Ra)}$ is noetherian over $\Comp Ra$.
The noetherianness of $\Gamma_{\m}(B)^{\vee}$ follows similarly.

Similarly, if $T$ is artinian, then
$\Gamma_{\fa}(T)^{\vee}$ is a noetherian $\comp R^{\fa}$-module while
$\Ga mT^{\vee}$ is a noetherian $\comp R^{\m}$-module. 
\end{disc}

\begin{prop}\label{xprop110217a}
Let $\fc$ be an intersection of finitely many maximal ideals of $R$. 
Let $L$ and $T$ be $R$-modules such that $T$ is $\fc$-torsion.
Let $\mathcal{F}$ be a finite subset of $\mspec(R)$ containing 
$\supp_R(T)\cap\ass_R(L)$, and set $\fb=\cap_{\m\in\mathcal{F}}\m$. 
Assume that there is an integer $x\geq 0$ such that $\fb^x\Gamma_{\fb}(L)=0$.
Set $y=\inf\{z\geq 0\mid \fb^zT=\fb^{z+1}T\}$, and let $n\geq\min\{x,y\}$. 
\begin{enumerate}[\rm(a)]
\item\label{xprop110217a1}
For each $\m\in\mathcal{F}$ 
there is an integer $\alpha_{\m}$ with $n\geq\alpha_{\m}\geq 0$ such that  $\m^{\alpha_{\m}}T=\m^{\alpha_{\m}+1}T$ or 
$\m^{\alpha_{\m}}\Gamma_{\m}(L)=0$. 
\item\label{xprop110217a2}
Given any $\alpha_\m$ as in part~\eqref{xprop110217a1}, there are $R$-module isomorphisms
\[
\textstyle\Hom[R]{T}{L}\cong\coprod_{\m\in\mathcal{F}}\Hom[R]{T/\m^{\alpha_{\m}}T}{(0:_L\m^{\alpha_{\m}})}
\cong\Hom[R]{T/\fb^nT}{(0:_L\fb^n)}.
\]
\end{enumerate}
\end{prop}

\begin{proof}
\eqref{xprop110217a1}
It suffices to show that $\m^{n}T=\m^{n+1}T$ or 
$\m^{n}\Gamma_{\m}(L)=0$ for each $\m\in\mcf$. 
To show this, we argue by cases.
If $n\geq x$, then we have $\m^n\Gamma_{\m}(L)=\fb^n\Gamma_{\m}(L)=0$
since $\Ga mL$ is an $R_{\m}$-module and $\fb R_{\m}=\m R_{\m}$.
In the case $n<x$, the condition $n\geq \min\{x,y\}$  implies that $n\geq y$. 
In particular, this implies that $\fb^nT=\fb^{n+1}T$.
Since $T=\coprod_{\n\in\supp(T)}\Gamma_\n(T)$, this explains the second equality in the next display
in the case $\m\in\supp_R(T)$:
\[
\m^n\Gamma_\m(T)=\fb^n\Gamma_\m(T)=\fb^{n+1}\Gamma_\m(T)=\m^{n+1}\Gamma_\m(T).
\]  
For $\m\neq\n\in\supp(T)$ we have $\m^j\Gamma_\n(T)=\Gamma_\n(T)=\m^{j+1}\Gamma_\n(T)$ for all $j\geq 0$. Thus
\begin{align*}
\m^nT&\textstyle=\left(\coprod_{\m\neq\n\in\supp(T)}\m^n\Gamma_\n(T)\right)\coprod\m^n\Gamma_\m(T)\\
&\textstyle=\left(\coprod_{\m\neq\n\in\supp(T)}\m^{n+1}\Gamma_\n(T)\right)\coprod\m^{n+1}\Gamma_\m(T)\\
&=\m^{n+1}T
\end{align*}
since $\supp_R(T)$ is finite. In the case $\m\notin\supp_R(T)$, we have $\Ga mT\cong T_{\m}=0$, so the displayed
equalities hold in this case as well.

\eqref{xprop110217a2}
For each integer $j\geq 0$, the first step in the following display is from 
Lemma~\ref{xcor28'} applied to $T\otimes_R(R/\fb^j)$:
\begin{align*}
T/\fb^jT
&\textstyle\cong \coprod_{\m\in\mcf}(T_{\m}/\fb^jT_{\m})
\textstyle\cong \coprod_{\m\in\mcf}(\Ga mT/\fb^j\Ga mT)
\cong\Ga bT/\fb^j\Ga bT.
\end{align*}
The second step is from Lemmas~\ref{lem110517a}\eqref{item120718c} and~\ref{lem120727a},
and the third step follows similarly.
This explains the third step in the next display:
\begin{align*}
\Hom[R]{T}{L}
&\cong\Hom[R]{\Gamma_\fb(T)}{\Gamma_\fb(L)}\\
&\cong\Hom[R]{\Gamma_\fb(T)/\fb^x\Gamma_{\fb}(T)}{\Gamma_\fb(L)}\\
&\cong\Hom[R]{T/\fb^xT}{\Gamma_\fb(L)}\\
&\cong\Hom[R]{T/\fb^nT}{\Gamma_\fb(L)}\\
&\cong\Hom[R]{T/\fb^nT}{(0:_L\fb^n)}.
\end{align*}
The first step is from Proposition~\ref{xlem:homSwap}.
The second step follows from the assumption
$\fb^x\Gamma_{\fb}(L)=0$.
The fifth step is due to the equality $(0:_L\fb^n)=(0:_{\Ga bL}\fb^n)$. 

For the fourth step, we argue by cases.
If $n\geq x$, then $\fb^n\Ga bL=0=\fb^x\Ga bL$, so we have
$\Hom[R]{T/\fb^xT}{\Gamma_\fb(L)}\cong\Hom[R]{T}{\Gamma_\fb(L)}\cong\Hom[R]{T/\fb^nT}{\Gamma_\fb(L)}$
as desired.
If $n< x$, then the condition $n\geq \min\{x,y\}$ implies that $y\leq n<x$. From the assumption $\fb^yT=\fb^{y+1}T$ it follows that $\fb^yT=\fb^nT=\fb^xT$.

Note that  for each $\m\in\mcf$ we have
\begin{equation*}
\m^x\Gamma_{\m}(L)=\fb^x\Gamma_{\m}(L)\subseteq\fb^x\Gamma_{\fb}(L)=0.
\end{equation*}
The first step is from the fact that $\Ga mL$ is an $R_{\m}$-module and $\fb R_{\m}=\m R_{\m}$.
The second step is from the fact that $\fb\subseteq\m$, and the vanishing is by the definition of $x$.
Similarly, for each $\m\in\mcf$, we have
$$\m^y T_{\m}=\fb^yT_{\m}=\fb^{y+1}T_{\m}=\m^{y+1}T_{\m}.
$$
Thus, we have the following isomorphisms
$$\Hom[R]{T}{L}
\textstyle\cong\coprod_{\m\in\mathcal{F}}\Hom[R]{\Ga mT}{\Gamma_{\m}(L)}
\cong\coprod_{\m\in\mathcal{F}}\Hom[R]{T/\m^{\alpha_{\m}}T}{(0:_L\m^{\alpha_{\m}})}$$
using similar reasoning as above.
\end{proof}

\begin{prop}\label{xthm29}
Let $\fc$ be an intersection of finitely many maximal ideals of $R$.
Let $L$ and $T$ be $R$-modules such that $T$ is $\fc$-torsion,
Let $\mathcal{F}$ be a finite subset of $\mspec(R)$
containing $\supp_R(T)\cap\ass_R(L)$, and set $\fb=\cap_{\m\in\mathcal{F}}\m$. 
Assume that there is an integer $x\geq 0$ such that $\fb^x\Gamma_{\fb}(L)=0$.
Set $y=\inf\{z\geq 0\mid \fb^zT=\fb^{z+1}T\}$, and let $n\geq\min\{x,y\}$. 
For each $\m\in\mathcal{F}$, 
fix an integer $\alpha_{\m}$ with $n\geq\alpha_{\m}\geq 0$ such that  $\m^{\alpha_{\m}}T=\m^{\alpha_{\m}+1}T$ or 
$\m^{\alpha_{\m}}\Gamma_{\m}(L)=0$. 
Then there are inequalities
\begin{align*}
\len_R(\Hom[R]{T}{L})&\textstyle\leq\sum_{\m\in\mathcal{F}}\len_R(T/\m T)\len_R(0:_L\m^{\alpha_{\m}})\\
&\leq\max\{\len_R(T/\m T)\mid \m\in\mathcal{F}\}\len_R(0:_L\fb^n)\\
&\leq\len_R(T/\fb T)\len_R(0:_L\fb^n).
\end{align*}
Here, we follow the convention $0\cdot\infty=0$.
\end{prop}

\begin{proof}
An inductive argument on  $\len_R(0:_L\m^{\alpha_{\m}})$ shows that
\[
\len_{R}(\Hom[R]{T/\m^{\alpha_{\m}}T}{(0:_L\m^{\alpha_{\m}})})\leq \len_R(T/\m T)\len_R(0:_L\m^{\alpha_{\m}}).
\]
Therefore by Proposition \ref{xprop110217a} and the additivity of length we get the first inequality in the proposition. 

The conditions $n\geq\alpha_{\m}$ and $\fb\subseteq\m$ for each $\m\in\mcf$ imply that
$\fb^n\subseteq\m^{\alpha_{\m}}$, so we have
$\textstyle
\sum_{\m\in\mcf}(0:_L\m^{\alpha_{\m}})\subseteq(0:_L\fb^n)$.
As   each $\m\in\mcf$ is maximal, the elements of $\mcf$ are comaximal in pairs, so the sum 
$\sum_{\m\in\mcf}(0:_L\m^{\alpha_{\m}})$ is direct.
It follows that
$$\textstyle\sum_{\m\in\mathcal{F}}\len_R(0:_L\m^{\alpha_{\m}})
=\len_R(\sum_{\m\in\mathcal{F}}(0:_L\m^{\alpha_{\m}}))\leq\len_R(0:_L\fb^n)$$
and the second in inequality in the statement of the proposition follows.
The third inequality in the statement of the proposition follows from the fact that $T/\fb T$ surjects onto $T/\m T$.
\end{proof}

\begin{para}[Proof of  Theorem~\ref{intthm120807b}]
\label{para120807h}
Set $\fb=\cap_{\m\in\mcg}\m$.

First, we show that $\Ga bN$ is annihilated by a power of $\fb$.
Since $N$ is noetherian, so is the submodule $\Ga bN$.
In particular, $\Ga bN$ is finitely generated. Since each generator 
of $\Ga bN$ is annihilated by a power of $\fb$, the same is true of $\Ga bN$.

Proposition~\ref{xprop110217a}\eqref{xprop110217a1} implies that 
for each $\m\in\mathcal{F}$ 
there is an integer $\alpha_{\m}$ with $n\geq\alpha_{\m}\geq 0$ such that  $\m^{\alpha_{\m}}A=\m^{\alpha_{\m}+1}A$ or 
$\m^{\alpha_{\m}}\Gamma_{\m}(N)=0$. 
Proposition~\ref{xprop110217a}\eqref{xprop110217a1} provides the isomorphism
$\textstyle\Hom[R]{A}{N}\cong\coprod_{\m\in\mathcal{G}}\Hom[R]{A/\m^{\alpha_{\m}}A}{(0:_N\m^{\alpha_{\m}})}$.
Since each module
$\Hom[R]{A/\m^{\alpha_{\m}}A}{(0:_N\m^{\alpha_{\m}})}$ is annihilated by $\m^{\alpha_{\m}}$, it follows that
$\Hom AN$ is annihilated by $\cap_{\m\in\mcg}\m^{\alpha_{\m}}$.

Proposition~\ref{xthm29} provides the first step in the next sequence:
$$\len_R(\Hom[R]{A}{N})\textstyle\leq\sum_{\m\in\mathcal{G}}\len_R(A/\m A)\len_R(0:_N\m^{\alpha_{\m}})<\infty.$$
For the second step,  observe that Lemma~\ref{xlem100213b} implies that
$A/\m A$ 
and $(0:_N\m^{\alpha_{\m}})$ both have finite length. 
\qed
\end{para}

\begin{defn}
Given an $R$-module $L$, a prime ideal $\p\in\spec(R)$ is an \emph{attached prime} of $L$ if there exists a
submodule $L'$ of $L$ such that $\p=\ann_R(L/L')$. The set of attached
primes of $L$ is denoted $\att_R(L)$.
\end{defn}

\begin{prop} \label{xprop:assHom}
Assume that $R$ is noetherian.
Let $A$ and $B$ be $R$-modules such that $A$ is artinian,
and set $\mathcal{F}=\supp_R(A)\cap\ass_R(B)$ and
$\fb=\cap_{\m\in\mathcal{F}}\m$. Assume that $\mu^0_R(\m,B)$ is finite for all $\m\in\mathcal{F}$. Then we have
\begin{align*}
\ass_{\comp R^{\fb}}(\Hom{A}{B})
&=\ass_{\comp R^\fb}(\Gamma_\fb(A)^{\vee})\cap\supp_{\comp R^\fb}(\Gamma_{\fb}(B)^{\vee})\\
&=\att_{\comp R^\fb}(\Gamma_\fb(A))\cap\supp_{\comp R^\fb}(\Gamma_{\fb}(B)^{\vee}).
\end{align*}
\end{prop}

\begin{proof}
Proposition~\ref{xlem:homSwap2} implies that 
$\Hom{A}{B}\cong\Hom[\comp R^\fb]{\Gamma_\fb(B)^{\vee}}{\Gamma_{\fb}(A)^{\vee}}$. Since $\Gamma_\fb(B)^{\vee}$ is a noetherian $\comp R^\fb$-module 
(by Remark~\ref{disc120729a})
we can apply a result of Bourbaki~\cite[IV 1.4 Proposition 10]{bourbaki:ac3-4} to obtain the first equality in the proposition. Also, by 
\cite[Proposition 2.7]{ooishi:mdwm}, we have the first equality in the next sequence:
$$\ass_{\comp R^\fb}(\Gamma_\fb(A)^{\vee})=\att_{\comp R^\fb}(\Gamma_\fb(A)^{\vee})^{\vee(\comp R^\fb)})=\att_{\comp R^\fb}(\Gamma_{\fb}(A)).$$
The second equality follows from the fact that $\Ga bA$ is artinian over the semi-local  ring $\Comp Rb$ by Fact~\ref{fact120720a}, hence Fact~\ref{xpara1}
implies that $\Ga bA$ is Matlis reflexive over $\Comp Rb$; so we have
$(\Gamma_\fb(A)^{\vee})^{\vee(\comp R^\fb)}\cong(\Gamma_\fb(A)^{\vee(\comp R^\fb)})^{\vee(\comp R^\fb)}\cong\Gamma_\fb(A)$ by Lemma~\ref{lem110413b}.
This explains  the
second equality in the proposition.
\end{proof}

\begin{cor}\label{xprop100320a}
Assume that $R$ is noetherian.
Let $A$ and $B$ be $R$-modules such that $A$ is artinian.  Set $\mathcal{F}=\supp_R(A)\cap\ass_R(B)$ 
and $\fb=\cap_{\m\in\mathcal{F}}\m$.
Assume that $\mu^0_R(\m,B)$ is finite for all $\m\in\mathcal{F}$.  Then the following conditions are equivalent:
\begin{enumerate}[\rm(i)]
\item \label{xprop100320a3'} 
$\Hom{A}{B}=0$;
\item \label{xprop100320a3} 
$\Hom{\Gamma_{\fb}(A)}{\Gamma_{\fb}(B)}=0$;
\item \label{xprop100320a7} 
$\Hom[\comp R^\fb]{\md{\Gamma_{\fb}(B)}}{\md{\Gamma_{\fb}(A)}}=0$;
\item \label{xprop100320a5} 
$\ass_{\comp R^\fb}(\Gamma_{\fb}(A)^{\vee})\cap\supp_{\comp R^\fb}(\md{\Gamma_{\fb}(B)})=\emptyset$; and
\item \label{xprop100320a9} 
$\att_{\comp R^\fb}(\Gamma_{\fb}(A))\cap\supp_{\comp R^\fb}(\md{\Gamma_{\fb}(B)})=\emptyset$.
\end{enumerate}
\end{cor}

\begin{proof}
Propositions~\ref{xlem:homSwap} and~\ref{xlem:homSwap2} give the equivalence of (\ref{xprop100320a3'})--(\ref{xprop100320a7}). 
The equivalence of (\ref{xprop100320a7})--(\ref{xprop100320a9}) follows from 
Proposition \ref{xprop:assHom} and the fact that we have $\Hom{A}{B}=0$ if and only if $\ass_{\comp R^{\fb}}(\Hom{A}{B})=\emptyset$.
\end{proof}


\providecommand{\bysame}{\leavevmode\hbox to3em{\hrulefill}\thinspace}
\providecommand{\MR}{\relax\ifhmode\unskip\space\fi MR }
\providecommand{\MRhref}[2]{%
  \href{http://www.ams.org/mathscinet-getitem?mr=#1}{#2}
}
\providecommand{\href}[2]{#2}

\end{document}